\documentclass[a4paper,fleqn]{cas-sc}

\usepackage[numbers]{natbib}

\def\tsc#1{\csdef{#1}{\textsc{\lowercase{#1}}\xspace}}
\tsc{WGM}
\tsc{QE}
\tsc{EP}
\tsc{PMS}
\tsc{BEC}
\tsc{DE}

\newtheorem{theorem}{Theorem}[section]
\newtheorem{lemma}{Lemma}[section]

\newtheorem{definition}{Definition}[section]

\usepackage{graphics,graphicx,color}
\usepackage{rotating}
\usepackage{amssymb,amsmath}
\usepackage{eucal}
\usepackage{algorithm}
\usepackage[noend]{algpseudocode}
\usepackage{epsfig,epstopdf}
\usepackage{geometry}
\usepackage{subcaption}
\usepackage[export]{adjustbox}

\newenvironment{proof}[1][Proof]{\noindent\textbf{#1.} }{\ \rule{0.5em}{0.5em}}

\begin{document}
\let\WriteBookmarks\relax
\def\floatpagepagefraction{1}
\def\textpagefraction{.001}
\shorttitle{Mathematical study of a new coupled electro-thermo radiofrequency model of cardiac tissue}
\shortauthors{Bendahmane, Ouakrim, Ouzrour and Zagour}

\title[mode = title]{
Mathematical study of a new coupled electro-thermo radiofrequency model of cardiac tissue} 

\author[1]{Mostafa Bendahmane}\cormark[1]
\address[1]{Institut de Math\'ematiques de Bordeaux and INRIA-Carmen Bordeaux Sud-Ouest, Universit\'e de Bordeaux, Bordeaux, France} 
\ead{mostafa.bendahmane@u-bordeaux.fr}

\author[2]{Youssef Ouakrim}
\address[2]{Laboratoire de Math\'ematiques, Mod\'elisation et Physique Appliqu\'ee, Ecole Normale Sup\'erieure de F\`es, Universit\'e Sidi Mohamed Ben Abdellah, Maroc.}
\ead{youssef.ouakrim@usmba.ac.ma}

\author[2]{Yassine Ouzrour}
\ead{yassine.ouzrour@usmba.ac.ma}
\author[3]{Mohamed Zagour}
\address[3]{Euromed University of Fes, UEMF, Morocco}
\ead{ zagourmohamed@gmail.com}

\begin{abstract}
This paper presents a  nonlinear  reaction-diffusion-fluid system that simulates radiofrequency ablation within cardiac tissue. 
The model conveys the dynamic evolution of temperature and electric potential in both the fluid and solid regions, along with the evolution of velocity within the solid region.
By formulating the system that describes the phenomena across the entire domain, encompassing both solid and fluid phases, we proceed to an analysis of well-posedness, considering a broad class of right-hand side terms. The system involves parameters such as heat conductivity, kinematic viscosity, and electrical conductivity, all of which exhibit nonlinearity contingent upon the  temperature variable.
The mathematical analysis extends to establishing the existence of a global solution, employing the Faedo-Galerkin method in a three-dimensional space. 
To enhance the practical applicability of our theoretical results, we complement our study with a series of numerical experiments. We implement the discrete system using the finite element method for spatial discretization and an Euler scheme for temporal discretization. Nonlinear parameters are linearized through decoupling systems, as introduced in our continuous analysis. These experiments are conducted to demonstrate and validate the theoretical findings we have established.
\end{abstract}
\maketitle
 \section{Introduction and problem statement}  
\subsection{Background}
The paper introduces a mathematical model for tracking the evolution of an invasive medical technique that is widely employed across various medical disciplines, specifically Radiofrequency Ablation (RFA).
Our goal is to improve the applicability and efficiency of this technique, particularly in the treatment of conditions like cardiac arrhythmia and the ablation of tumors located in different regions of the body.
The RFA technique is characterized by its exceptional precision, high effectiveness, and low mortality rates. 
RFA has evolved into a widely accepted and highly effective treatment for a variety of cardiac arrhythmias, including ventricular arrhythmias, atrial fibrillation, and atrial tachycardia, among others.
The operation of this treatment technique is based on the application of high-frequency electrical current within specific myocardial regions (see Figure \ref{Fig:illustration}).
This process leads to the generation of elevated temperatures (typically exceeding $50\, ^{\circ}$C) within the cardiac tissue, resulting in cell death.
RFA procedures are mathematically represented through a thermistor problem, which simulates the heating of a conductive material by inducing electric current at a specific boundary region for a defined time period. This model is formulated as a coupled system of nonlinear partial differential equations (PDEs), specifically consisting of the heat equation with Joule heating as the heat source and the current conservation equation with temperature-dependent electrical conductivity.

The need to control the results of RFA experiments has driven various research efforts. 
For instance, there has been a focus on predicting tissue temperatures to provide real-time guidance during the process, as established in~\cite{johnson2002,villard2005}. 
These investigations have primarily centered on optimal control and inverse problems, with a specific emphasis on identifying the frequency factor and energy involved in the thermal damage function for different tissue types. In this context, we refer to the two works by Meinlschmidt et al.~\cite{MEINLSCHMIDT-part1,MEINLSCHMIDT-part2}, which addressed boundary condition identification problems through optimization techniques. The authors established the well-posedness and optimality conditions.
A recent contribution by Huaman et al.~\cite{huaman2023local} explores the local null controllability of the thermistor problem,  considering spatially distributed control.

In the following subsections, we describe the coupled system that models the dynamics of radiofrequency ablation (RFA) treatment in the presence of a specific fluid, namely, blood. 
Let $T>0$ represent the final time, and consider bounded open subsets in three-dimensional space denoted as $\Omega_{b}$ and $\Omega_{ts}$. 
These subsets correspond respectively to the blood vessel and the cardiac tissue and possess piecewise smooth boundaries denoted by $\Sigma_{b}$ and $\Sigma_{ts}$. 
For visual reference, Figure \ref{Fig:illustration}  illustrates the radiofrequency ablation procedure within cardiac tissue, showing an overview of different regions within the domain. The readers can consult the configuration geometry and boundaries conditions of our model in Figure \ref{Domain}.
\begin{figure}[pos=!ht]
	\begin{center}
		\includegraphics[width=.9\linewidth
		]
		{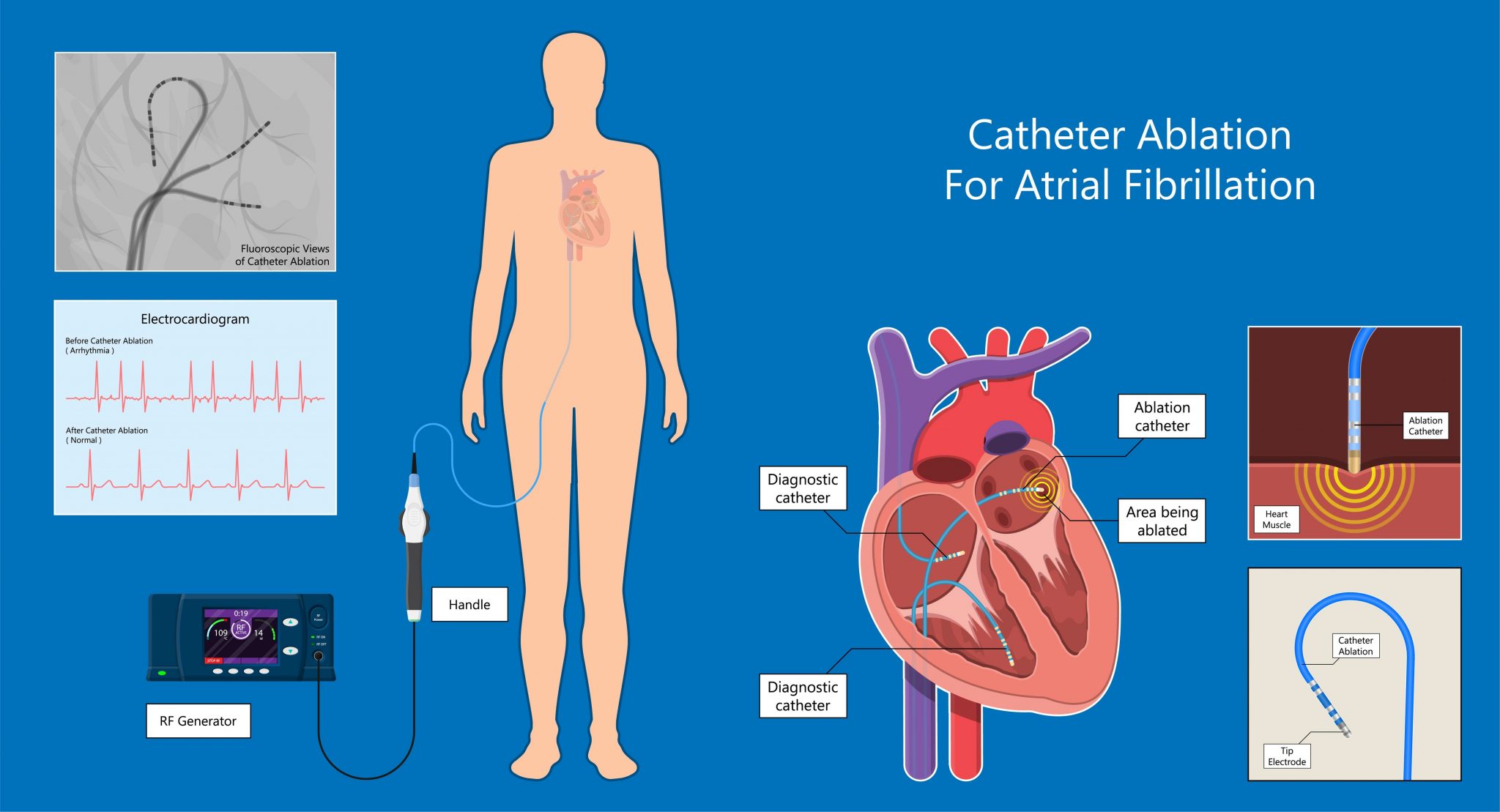}
	\end{center}
	\caption{Illustration of radiofrequency ablation procedure within cardiac tissue, highlighting different regions in the domain. \url{https://www.stopafib.org/procedures-for-afib/catheter-ablation/}
	}
	\label{Fig:illustration} 
\end{figure}
\subsection{Electrothermal field in thermistor models}

RFA utilizes alternating electric currents within the radiofrequency range ($450$ - $500$ kHz) to induce controlled thermal damage in tissues. This technique is commonly employed for the treatment of conditions such as liver, lung, and kidney tumors, varicose veins, as well as cardiac arrhythmias. During this procedure, an electrode is inserted under image guidance to precisely target and heat the tissue.
When alternating electric fields are applied to resistive materials like biological tissue, two forms of heating occur; conduction losses resulting from resistive heating due to ion movement and dielectric losses caused by the rotation of molecules in the alternating electric field. However, within the frequency range below $1~$MHz, dielectric losses become negligible \cite{Berjano2006}.

\noindent In our model, we exclusively focus on resistive heating. The resulting electric field within the tissue can be accurately described by the Laplace equation:
\begin{equation}
\nabla \cdot (\sigma_{ts}\nabla \varphi_{ts}) = 0 \quad \text{in } \Omega_{T,ts} := (0, T) \times \Omega_{ts},
\end{equation}
where, $\sigma_{ts}$ denotes the electrical conductivity of the material in Siemens per meter $(\mathrm{S/m})$, and $\varphi_{ts}$ represents the electric potential in Volts $(V)$. 
The electric field intensity, denoted as $\mathbf{E}$ in units of Volts per meter $(\mathrm{V/m})$, and the current density $\mathbf{J}$ in Amperes per square meter $(\mathrm{A/m}^2)$ are calculated from the following relationships: $\mathbf{E} = -\nabla \varphi$ and $\mathbf{J} = \sigma_{ts} \mathbf{E}$.
The local power density responsible for tissue heating is determined by multiplying the current density $\boldsymbol{J}$ by the electric field intensity $\boldsymbol{E}$. This power density is then used to compute the temperature distribution within the tissue using the heat-transfer equation \cite{Dieter2010}.
\\
This equation accurately characterizes the situation when the ablation catheter is within the tissue, although it's important to acknowledge that the electric field might also extend into the adjacent blood vessel. As a result, both scenarios are considered simultaneously.

\subsection{Bio-heat distribution}
During radiofrequency ablation (RFA), electrical energy is delivered, causing a rise in temperature in both the blood vessel and the cardiac tissue surrounding the catheter. In our study we employs the Pennes equation to describe the bio-heat transfer model:
\begin{equation}\label{bio-heat-system}
	\begin{split}
		\rho_{1,i}\rho_{2,i}\,\theta_{t,i} - \nabla \cdot(\eta_i(\theta_i) \nabla \theta_i) &=\mathcal{S}_i+\rho_{m,i}\qquad \hbox{ in } \,\Omega_{T,i}:=(0,T)\times \Omega_i,\qquad\text{for $i=b,\,ts$},
	\end{split}
\end{equation}
where $\rho_{1,i}$ ($kg/m^3$), $\rho_{2,i}$ ($J/kg\times K$), and $\eta_i$ represent the density, specific heat, and thermal conductivity functions dependent on the media (blood vessel for $i=b$ or the tissue for $i=ts$), respectively. 
Recall that $\Omega_{b}$ and $\Omega_{ts}$ represent the tissue and blood vessel domains. The coefficient $\rho_{m,i}$ models the metabolic heat generation for $i=b,ts$. 
Finally, the electromagnetic heat source $\mathcal{S}_i$ ($W/m^3$) due to radiofrequency heating is given by:
\begin{equation}\label{eq-S}
	\mathcal{S}_i=\sigma_i(\theta_i)|{\nabla \varphi_i}|^2 \qquad \text{ in }\Omega_{T,i},
\qquad\text{for $i=b,\,ts$}.
\end{equation}
It's important to mention that during RFA, the electromagnetic heat source ($\mathcal{S}\simeq \mathcal{O}(10^8)$) is significantly greater in magnitude compared to the metabolic heat source ($\rho_m \simeq \mathcal{O}(10^3)$). Therefore, in this work, we neglect the metabolic heat generation $\rho_{m,i}$ for $i=b,ts$. 
In our bio-heat model \eqref{bio-heat-system}, the thermal and electrical conductivities of the blood and the tissue depend on the temperature of the media.

\subsection{Blood flow model}
The blood flow can be characterized as an incompressible Navier–Stokes fluid, governed by the following system:
\begin{equation}
    \begin{array}{rclll}
        \rho\left(\partial_t \boldsymbol{u}+(\boldsymbol{u}\cdot \nabla) \boldsymbol{u}\right)
        -\nabla \cdot(\mu(\theta_b) \mathbb{D}(\boldsymbol{u}))+\nabla \pi &=& \boldsymbol{F} & \hbox { in } &\Omega_{T,b}, \\
        \nabla \cdot \boldsymbol{u} &=& 0 & \hbox { in } &\Omega_{T,b}.
    \end{array}
    \label{1.2} 
\end{equation}
Here, $\boldsymbol{u}$ represents the flow velocity, and $\pi$ denotes the pressure scaled by the density $\rho$. The parameter $\mu$ signifies the dynamic viscosity of the blood, equivalent to $\rho\nu$, where $\nu$ is the kinematic viscosity that depends on the temperature $\theta_b$ of the blood vessel. 
In  \eqref{1.2}, the functions $\mathbb D(\boldsymbol{u}) =\frac 1 2 \left( \nabla u + \nabla u ^T  \right) $ and $\boldsymbol{F}$ represents the strain rate tensor and the external force, respectively.\\
\noindent In the context of initial data for fluid velocity, it is essential to provide a carefully prescribed value. Typically, this velocity field should be divergence-free to be considered admissible. In many hemodynamic simulations, the actual value of this quantity is often unknown. As a result, it is frequently set to zero throughout the domain, or, in a more informed approach, it is estimated as the solution to a stationary Stokes problem.
The issue of boundary conditions is of utmost significance when simulating blood flow. Extensive literature has been dedicated to addressing this topic in recent years, with comprehensive reviews available, such as in \cite{FI14,QMV17}.

\subsection{Electro-thermo-fluid model}
The electro-thermo-fluid model addressed in this paper is formulated as follows, summarizing all the equations described above:\begin{equation}\label{S1}
		\left\{
	\begin{split}
		\rho\left(\partial_t \boldsymbol{u}+(\boldsymbol{u}\cdot \nabla) \boldsymbol{u}\right)
		-\nabla \cdot(\mu(\theta_b) \mathbb{D}(\boldsymbol{u}))+\nabla \pi &= \boldsymbol{F},   \,\,\,\quad \qquad \qquad \hbox { in } \Omega_{T,b}, \\\\
		\nabla \cdot \boldsymbol{u} &= 0,\qquad \qquad \qquad \hbox { in } \Omega_{T,b},\\\\
		\partial_t \theta_b - \nabla \cdot(\eta_b(\theta_b) \nabla \theta_b) + \boldsymbol{u} \cdot \nabla \theta_b
		-(\sigma_b(\theta_b)\nabla \varphi_{b})\cdot \nabla \varphi_{b} &=0, \qquad \qquad \qquad \hbox { in }  \Omega_{T,b}  \\\\
		- \operatorname{div}(\sigma_b(\theta_b) \nabla \varphi_{b})&=0, \qquad \qquad \qquad \hbox { in } \Omega_{T,b} ,\\\\
		\partial_t \theta_{ts} - \nabla \cdot(\eta_{ts}(\theta_{ts}) \nabla \theta_{ts}) 
		-(\sigma_{ts}(\theta_{ts})\nabla \varphi_{ts})\cdot \nabla \varphi_{ts} &=0, \qquad \qquad \qquad \hbox { in }  \Omega_{T,ts}  \\\\
		- \operatorname{div}(\sigma_{ts}(\theta_{ts}) \nabla \varphi_{ts})&=0, \qquad \qquad \qquad \hbox { in } \Omega_{T,ts} .
	\end{split}
	\right.
\end{equation} 
We complete the system \eqref{S1} with the following boundary conditions 
\begin{equation}\begin{split}\label{S2} 
		&\boldsymbol{u}=\boldsymbol{u}_d \,\,  \hbox { on } \Sigma_{T,b}, 
		\\
		&\theta_b=\bar{\theta} \,\,  \hbox { on } \Sigma^D_{T,b}\quad \text{and}\quad \theta_{ts}=\bar{\theta} \,\,  \hbox { on } \Sigma^D_{T,ts},\\ 
		&\theta_b=\theta_{ts} \,\,  \hbox { on } \Sigma_{T,7}\cup \Sigma_{T,8}
		\quad \text{and}\quad (\eta_b(\theta_b) \nabla \theta_b) \cdot \boldsymbol{n}_b=-(\eta_{ts}(\theta_{ts})\nabla \theta_{ts}) \cdot \boldsymbol{n}_{b}\,\,  \hbox { on } \Sigma_{T,7}\cup \Sigma_{T,8},\\
		& \sigma_i(\theta_i) \nabla \varphi_i \cdot \boldsymbol{n}_b=0 \,\,\hbox { on }\Sigma_{T,7} \quad \text{and}
		\quad \varphi_i=\varphi_{i,d}  \,\,\hbox { on }  \Sigma_{T,i}\setminus \Sigma_{T,7},
	\end{split}
\end{equation}
where $\Sigma_{T,i}:=(0,T)\times \Sigma_i$ for $i=b,ts$. The initial data are 
\begin{equation}\label{S3}
	\begin{split}
		&\theta_i(0,.)=\bar{\theta} \quad \text{ on }  \quad \Omega_i, \qquad  \forall i=b,ts. \\
		&   \boldsymbol{u}(0,.)=\boldsymbol{u}_0 \quad \text{ on } \quad \Omega_b.
	\end{split}
\end{equation}  
Here, we divide the boundaries of blood domain $\Omega_b$ and tissue domain $\Omega_{ts}$ into regular parts $\Sigma^D_i$, $\Sigma^N_i$ and $\Sigma_8$, as illustrated in Figure \ref{Domain}. Thus,
\begin{align*}
	&\displaystyle \Sigma_i:= \Sigma^D_{i}\cup \Sigma_{7}\cup \Sigma_8,
	\quad \Sigma^D_b:= \cup_{i=0}^3 \Sigma_{i},\quad
	\displaystyle \Sigma^D_{ts}:=\cup_{i=4}^6 \Sigma_{i},
\end{align*}
and
\begin{equation*}
\begin{split}
&\theta_{d}=
\begin{cases}
	&\theta_s \quad  \hbox {on } \Sigma_{T,8},\\
	&\bar{\theta}  \quad \,\,\hbox { on }  \Sigma_{T}\setminus (\Sigma_{T,7}\cup \Sigma_{T,8}).
\end{cases}, \quad
\varphi_{i,d}=
\begin{cases}
	&\varphi_d \quad  \hbox {on } \Sigma_{T,8},\\
	&0  \quad \,\,\hbox { on }  \Sigma_{T,i}\setminus (\Sigma_{T,7}\cup \Sigma_{T,8}).
\end{cases}
\\ 
&\qquad \text{and}\quad\qquad  \boldsymbol{u}_d=
\begin{cases}
	&\boldsymbol{u}_e \qquad \,\,\text{on}\quad \text{$\Sigma_{T,1}\cup \Sigma_{T,3}$} ,\\
	&\boldsymbol{u}_s  \qquad \,\,\text{on}\quad \text{$\Sigma_{T,8}$},\\
	&\mathbf{0} \qquad\quad  \text{on}\quad \text{$\Sigma_{7}\cup \Sigma_{T,2}$},
\end{cases}
\end{split}
\end{equation*}
\begin{figure}
	\begin{center}
		\includegraphics[width=.9\linewidth]{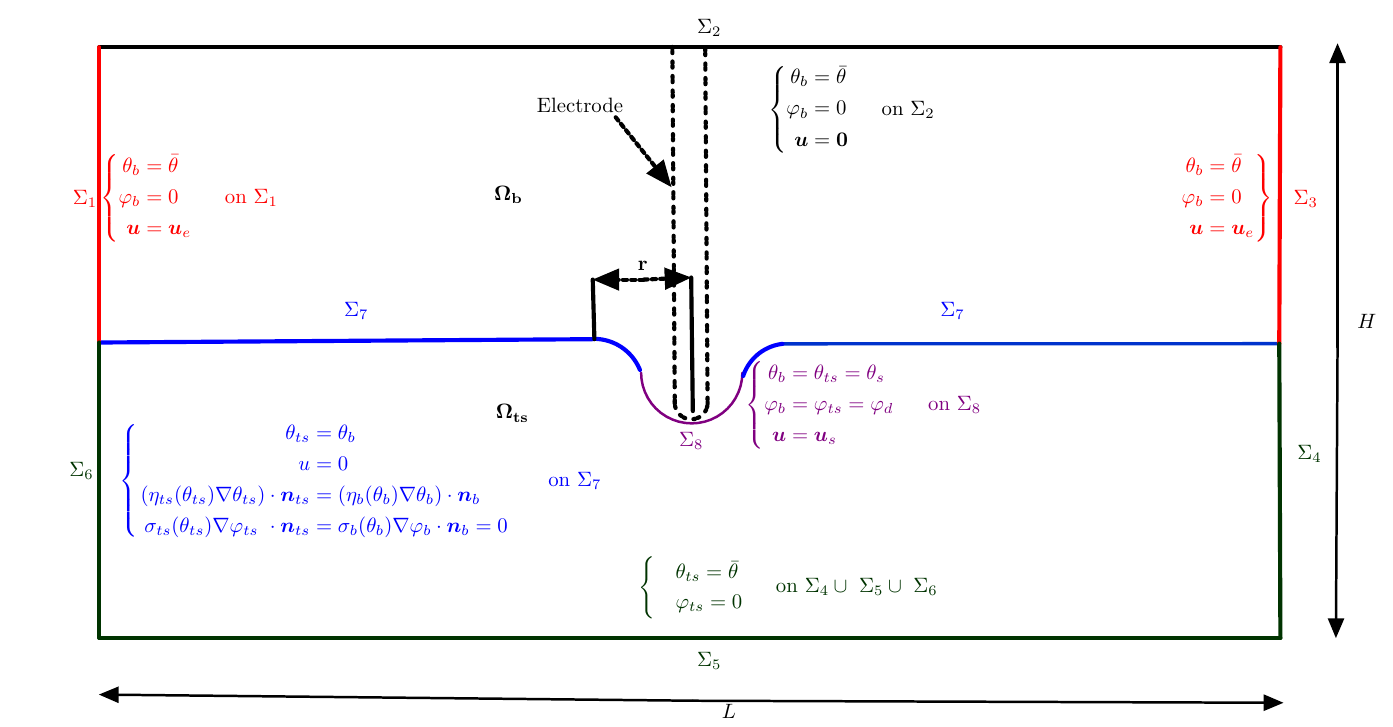}
	\end{center}
	\caption{Configuration geometry and boundaries conditions of the model.}
	\label{Domain}
\end{figure} 
Regarding the blood boundary conditions, 
we impose inlet and outlet velocity boundary conditions $\boldsymbol{u}=\boldsymbol{u}_e$ on the left and right surfaces of the blood volume ($\Sigma_{T,1}\cup \Sigma_{T,3}$). We assume that the saline irrigation flow ($\boldsymbol{u}=\boldsymbol{u}_s$ on $\Sigma_{T,8}$) serves as a velocity boundary condition applied to the area of the electrode tip. On the remaining boundary of the blood volume $\Sigma_b$ ($\Sigma_{7}$), we apply a no-slip condition ($\boldsymbol{u}=0$). 
For the thermal boundary conditions, a constant temperature $\bar{\theta} = 37^\circ C$ is set on $\Sigma^D_{T,b}\setminus(\Sigma_{T,7}\cup \Sigma_{T,8})$ for $\theta_b$ {and temperature of the saline flow is denoted by $\theta_s$ ($\theta_{ts}=\theta_b=\theta_s$ on  $\Sigma_{T,8}$)}. On the contact surface ($\Sigma_{T,7}$) between the two media (blood and tissue), we ensure the continuity of the flux and temperature. Similarly, a constant temperature of $\theta_{ts} = 37^\circ C$ is maintained on the remaining boundary surface of the tissue ($\Sigma_{T,4}\cup \Sigma_{T,5} \cup \Sigma_{T,6}$).
In the electrical model, we enforce a zero flux boundary condition on the tissue-blood surface $\Sigma^N_b$, except on the surface of the conducting part of the RF probe ($\Sigma_{T,8}$), where $\varphi_i=\varphi_d$ for $i=b,ts$. On the remaining boundary surface of the tissue and blood, we set $\varphi_i=0$ for $i=b,ts$.

\subsection{Contributions and related works}

The mathematical study of thermistor models, particularly the heat-potential model, has been the subject of several works. For instance, in~\cite{Allegretto1992}, the existence of a solution is established based on the maximum principle and a fixed-point argument. In~\cite{ref2}, the existence of a weak solution over an arbitrarily large time interval is demonstrated using the Faedo-Galerkin method. In~\cite{Li_Yang2015}, the existence and uniqueness of the solution for thermistor equations were addressed without relying on non-generalized assumptions on the data. Additionally, in the case of constant thermal conductivities, the existence and uniqueness of the solution in three-dimensional space are shown, with a particular emphasis on its $\alpha$-H\"{o}lder continuity, as seen in~\cite{Yuan1994}.
From a computational perspective, there is currently a lack of numerical analyses covering various scenarios of the model, including variations in conductivity and spatial considerations. In the context of two-dimensional space, where constant thermal conductivity is assumed, significant results were developed in \cite{Akrivis2005}. In the latter, by using a mixed finite element method with a linearized semi-implicit Euler scheme, an optimal error estimate for the $L^2$ norm is derived while adhering to specific time-step conditions.
Expanding this computational approach to higher-dimensional spaces, particularly in two and three dimensions (see \cite{Li2012}), time-step conditions for linearized semi-implicit methods are developed. Furthermore, optimal error estimates have been presented for a Crank-Nicolson Galerkin method concerning nonlinear thermistor equations~\cite{Li2014}, as well as for similar schemes based on backward differential formulas~\cite{Gao2015}.
In order to enhance our understanding of the practical application of thermistor equations in various scenarios and dimensions, developing models that couple the thermistor system with Navier-Stokes systems presents a challenge. These results can be particularly relevant for addressing fluid dynamics during ablation, as evident in blood flow~\cite{fang2022radiofrequency, materiel1, linte2018lesion,   petras2019computational,  sanchez2022computer, vaidya2021simulation, zhang2016review}. Some mathematical and mechanical studies of blood flow dynamics have been the subject of research,  such as \cite{FQV09, QMV17}.

In this paper, our primary focus is on the mathematical investigation and numerical approximation of a Thermo-Electric-Flow model. What distinguishes this model is its method of partitioning the domain into two distinct regions, representing the blood vessel and cardiac tissue, as visually depicted in Figure \ref{Domain}.
We undertake a rigorous analysis of the electro-thermo-fluid model in three dimensions, introducing an equivalent variational formulation. Our choice of temperature-dependent electrical and thermal conductivities ensures that they satisfy the necessary conditions for the existence of solutions. By employing the Faedo-Galerkin approach and  using compactness arguments, we are able to establish the existence of solutions for this model.
However, it's important to note that the practical implementation of our model poses certain challenges. 
Specifically, the variational formulations of the heat equation encounter complications due to the quadratic terms involving electric fields in \eqref{S1}, denoted as $(\sigma_i(\theta_i)|\nabla \varphi_i|^{2}$), where $i=b,ts$. Since we are seeking $\varphi_i$ in the space $H^1(\Omega_i)$, it follows that $\sigma_i(\theta_i)|\nabla \varphi_i|^{2}$ is limited to $L^{1}(\Omega_i)$. Several approaches can be employed to address this issue, including the selection of test functions in the variational formulations of temperatures within the space $H^{1}(\Omega_i)\cap L^{\infty}(\Omega_i)$. For further details on this approach, please refer to~\cite{bernardi2018finite}. However, in our case, we opt for the method proposed in~\cite{antontsev1994thermistor} by designating the boundary condition $\varphi_d$ within the space $H^{1}(\Omega_i)\cap L^{\infty}(\Omega_i)$.
The variational formulation is subsequently discretized using finite element schemes in both domains, and we utilize the domain decomposition method, particularly the Schwarz algorithm, to ensure the continuity of heat and potential equations. Temporal discretization is implemented through an Euler scheme. Finally, we present the results of numerical experiments conducted for various scenarios.

\subsection{Organization of the paper}
The structure of this paper is as follows. In the next section, we will introduce fundamental notations and appropriate functional spaces. Then, we will formulate the problem within a variational framework and present the existence result for the proposed model. Section \ref{proof-thm1} is dedicated to the proof of this result. We will also outline a two-dimensional numerical approach, which will be illustrated in Section \ref{numerical} through various numerical experiments involving different parameter values. Lastly, Section \ref{conclusion} will consolidate our findings and provide insights into future perspectives.

\section{Notion of solution and main result}      

\subsection{Mathematical setting}
Before presenting our results regarding weak soluions, we will first provide some preliminary materials, including relevant notations and conditions imposed on the data. \\ 
Let $\Omega=\Omega_{b}\cup \Omega_{ts}$ be a bounded open subset of $\mathbb{R}^N$, where $N=3$, such that $\Omega_{b} \cap\Omega_{ts}=\varnothing$. The boundary of $\Omega$ is smooth and can be denoted as $\partial \Omega=\Sigma_{b}\cup \Sigma_{ts}$, with $\Sigma_{b} \cap\Sigma_{ts}=\Sigma_7\cup \Sigma_8$ (see Figure \ref{Domain}). The symbol $|\Omega|$ represents the Lebesgue measure of $\Omega$. 
For $i=ts, b$, we denote $H^1(\Omega_i)$ as the Sobolev space of functions $\psi:\Omega_i \to \mathbb{R}$, where $\psi\in L^2(\Omega_i)$ and $ \nabla \psi \in L^2(\Omega_i ;\mathbb{R}^N)$. 
The notation $\parallel \cdot \parallel_{L^p(\Omega_i)}$ represents the standard norm in $L^p(\Omega_i)$, with $1\leq p \leq +\infty$: 
$$
L^{p}(\Omega_i)=\{u:\Omega \longrightarrow  {{\mathbb{R}}}\;\mbox{measurable} 
\mbox{~and~} \int_{\Omega_i} |u(x)|^{p}dx<+\infty\},
$$
$$
L^{\infty}(\Omega_i)=\{u:\Omega_i \longrightarrow
{{\mathbb{R}}}\;\mbox{measurable}
\mbox{~and~} \sup_{x \in \Omega_i}|u(x)|<+\infty\}.
$$ 
If $X$ is a Banach space, $a < b$, and $1 \leq p \leq +\infty$, the notation $L^p(a,b;X)$ represents the space of all measurable functions $\psi: (a,b) \to X$ such that $\parallel \psi(\cdot) \parallel_X$ belongs to $L^p(a,b)$. Note that $C^c_0(0, T; X)$ denotes the space of continuous functions with compact support and values in $X$.
\\
To simplify mathematical formulations, we introduce the following notations, assuming that the functions $\boldsymbol{v}$, $\boldsymbol{w}$, $\boldsymbol{\psi}$, $\theta_i$, $\phi_i$, $\varphi_i$, $\chi_i$, and $S_i$ for $i = b, ts$ are sufficiently smooth so that the following integrals are well-defined:
\begin{align*} 
	(\boldsymbol{w}, \boldsymbol{\psi})&=\int_{\Omega_b} \boldsymbol{w} \cdot \boldsymbol{\psi} dx 
	, \qquad
	&   (\theta_i, S_i)_{\Gamma_i}&=\int_{\Gamma_i}   \theta_i \, S_i \, \mathrm{d} \zeta,
	\\
	\tilde{a}_{\boldsymbol{w}}(\boldsymbol{w}, \boldsymbol{\psi})&=\int_{\Omega_b} \mathbb{D}(\boldsymbol{w}): \nabla\boldsymbol{\psi} dx
	, \qquad	
	& 	d(\boldsymbol{w}, \theta_i, S_i)&=\int_{\Omega_i} (\boldsymbol{w} \cdot \nabla \theta_i) S_i dx,
	\\
	b(\boldsymbol{w}, \boldsymbol{v}, \boldsymbol{\psi})&=\int_{\Omega_b}(\boldsymbol{w}.\nabla ) \boldsymbol{v}\cdot\boldsymbol{\psi} dx 
	, \qquad
	& 	a_{\theta_i}(\phi_i ; \theta_i, S_i)&=\int_{\Omega_i} \eta_i(\phi_i) \nabla \theta_i \cdot \nabla S_i dx ,
	\\
	\tilde{a}_{\theta_i}(\theta_i, S_i )&=\int_{\Omega_i} \nabla \theta_i \cdot \nabla S_i  dx 
	, \qquad
	& 	c_{\varphi_i}(\phi_i ,\varphi_i,S_i )&=\int_{\Omega_i} \sigma_i(\phi_i) \nabla \varphi_i \cdot \nabla \varphi_i S_i dx,
	\\ 
	a_{\varphi_i}(\phi_i ,\varphi_i,\chi_i )&=\int_{\Omega_i} \sigma_i(\phi_i) \nabla \varphi_i \cdot \nabla \chi_i dx
	, \qquad
	& a_{\boldsymbol{w}}(\phi ; \boldsymbol{\boldsymbol{w}}, \boldsymbol{\psi})&=\int_{\Omega_b} \nu(\phi) \mathbb{D}(\boldsymbol{w}):\nabla\boldsymbol{\psi} dx .
\end{align*} 
Furthermore, in the analysis of our RF-ablation model \eqref{S1}, we will utilize the following vectorial spaces:
\begin{equation*}
	\begin{split}
		&\mathbf{L}^2(\Omega_b)=(L^2(\Omega_b))^3, \qquad \mathbf{H}^1(\Omega_b)=(H^1(\Omega_b))^3, \mathbf{H}_{0}^{1}(\Omega_b)=\left(H_{0}^{1}(\Omega_b)\right)^3,\\
		&\boldsymbol{\mathcal{H}}^{\boldsymbol{u}}_{\boldsymbol{0}}= \left\{\boldsymbol{u}\in  \mathbf{H}^1(\Omega_b), \;  \operatorname{div} \boldsymbol{u}=0\, \text{ and }\,\boldsymbol{u}=\boldsymbol{0} \,\,  \hbox { on } \Sigma_{b} \right\} ,\quad \mathcal{L}_{\boldsymbol{b}}= \overline{\boldsymbol {\mathcal{H}}^{\boldsymbol{u}}_{\boldsymbol{0}}}^{\mathbf{L}^2(\Omega_b)},\\
		&\mathcal H^{\theta_{i}}_{0}=\left\{\theta_i\in  H^1(\Omega_i), \;  \theta_i=0 \,\,  \hbox { on } \Sigma^D_{T,i}\right\}, \hbox{ for  } i=b,ts\\
		&\mathcal H^{\varphi_{i}}_{0}=\left\{\varphi_i\in  H^1(\Omega_i), \;  \varphi_i=0\ \,\,  \hbox { on }\Sigma_{T,i}{\small \setminus} \Sigma_{T,7} \right\} , \hbox{ for  } i=b,ts.
	\end{split}
\end{equation*}
The properties of the forms $b(\cdot, \cdot, \cdot)$ are described in the following lemma, where $C_1$ is a positive constant (dependent on the domain) that may vary between different lines.

\begin{lemma}[ Properties of the trilinear form \cite{boyer2012mathematical}] The trilinear form b is continuous on $\mathbf{H}_{0}^{1}(\Omega) \times\mathbf{H}_{0}^{1}(\Omega) \times\mathbf{H}_{0}^{1}(\Omega)$ and satisfies:
	$$   
	\begin{aligned}
		b(\boldsymbol{w}, \boldsymbol{v}, \boldsymbol{\psi})+b(\boldsymbol{w}, \boldsymbol{\psi}, \boldsymbol{v}) &=0, \forall \boldsymbol{w}, \boldsymbol{v}, \boldsymbol{\psi} \in \boldsymbol {\mathcal{H}}^{\boldsymbol{u}}_{\boldsymbol{0}}, \\
		b(\boldsymbol{w}, \boldsymbol{v}, \boldsymbol{v}) &=0, \forall \boldsymbol{w}, \boldsymbol{v} \in \boldsymbol {\mathcal{H}}^{\boldsymbol{u}}_{\boldsymbol{0}}.
	\end{aligned} 
	$$
	Moreover, for all $\boldsymbol{w}, \boldsymbol{v}, \boldsymbol{\psi} \in \boldsymbol {\mathcal{H}}^{\boldsymbol{u}}_{\boldsymbol{0}},$ we have
	$$
	|b(\boldsymbol{w}, \boldsymbol{v}, \boldsymbol{\psi})| \leq C_1\|\boldsymbol{w}\|_{L^{2}}^{1/ 4}\|\boldsymbol{w}\|_{H^{1}}^{3 / 4}\|\boldsymbol{v}\|_{L^{2}}^{1/ 4}\|\boldsymbol{v}\|_{H^{1}}^{3 / 4}\|\boldsymbol{\psi}\|_{H^{1}}.
	$$
\label{pro-trilinear}
\end{lemma}
For all $\boldsymbol{w}, \boldsymbol{v} \in \boldsymbol {\mathcal{H}}^{\boldsymbol{u}}_{\boldsymbol{0}}$ we denote as $B(\boldsymbol{w}, \boldsymbol{v}) \in (\boldsymbol {\mathcal{H}}^{\boldsymbol{u}}_{\boldsymbol{0}})^{\prime}$ the bilinear form on $\boldsymbol {\mathcal{H}}^{\boldsymbol{u}}_{\boldsymbol{0}}\times\boldsymbol {\mathcal{H}}^{\boldsymbol{u}}_{\boldsymbol{0}}$ defined by
$$
\langle B(\boldsymbol{w}, \boldsymbol{v}), \boldsymbol{\psi}\rangle_{(\boldsymbol {\mathcal{H}}^{\boldsymbol{u}}_{\boldsymbol{0}})^{\prime}, \boldsymbol {\mathcal{H}}^{\boldsymbol{u}}_{\boldsymbol{0}}}=b(\boldsymbol{w}, \boldsymbol{v}, \boldsymbol{\psi}) .
$$
Then, by Lemma \ref{pro-trilinear}, it is clear that the map $B$ is continuous from $\boldsymbol {\mathcal{H}}^{\boldsymbol{u}}_{\boldsymbol{0}} \times \boldsymbol {\mathcal{H}}^{\boldsymbol{u}}_{\boldsymbol{0}}$ into $(\boldsymbol {\mathcal{H}}^{\boldsymbol{u}}_{\boldsymbol{0}})^{\prime}$ and that we have
\begin{equation}\label{bounded-trilniear}
	\|B(\boldsymbol{w}, \boldsymbol{w})\|_{(\boldsymbol {\mathcal{H}}^{\boldsymbol{u}}_{\boldsymbol{0}})^{\prime}} \leq C_1\|\boldsymbol{w}\|_{L^{2}}^{1/ 2}\|\boldsymbol{w}\|_{H^{1}}^{3 / 2},\quad \forall \boldsymbol{w} \in \boldsymbol {\mathcal{H}}^{\boldsymbol{u}}_{\boldsymbol{0}}.
\end{equation} 


\subsection{Variational formulation}     
Before presenting our main result, we make several assumptions about the parameters and data of the model \eqref{S1}-\eqref{S3}. Without further references, we assume the following assumptions:\\
Firstly, we assume that the functions $\mu$, $\sigma_i$, and $\eta_i$ are positive, bounded, and continuous with respect to the temperature
\begin{eqnarray}\label{assump-1}
	\begin{split}
		0<\underline{\nu} \leq \nu(\xi) \leq \overline{\nu}<+\infty \qquad & \forall \xi \in \mathbb{R},\\	
		0<\underline{\sigma}_{i} \leq \sigma_i(\xi) \leq \overline{\sigma}_{i}<+\infty \qquad & \forall \xi \in \mathbb{R} \,\text{ and } i=b,ts \\
		0<\underline{\eta}_{i} \leq \eta_i(\xi) \leq \overline{\eta}_{i}<+\infty \qquad &  \forall \xi \in \mathbb{R} \,\text{ and } i=b,ts 
	\end{split} 
\end{eqnarray}
where 
$\underline{\nu}$, $\overline{\nu}$, $\underline{\sigma}_{i}$, $\overline{\sigma}_{i}$, $\underline{\eta}_{i}$ and  $\overline{\eta}_{i}$  are positive constants. \\
In addition, with respect to the boundary conditions of \eqref{S2}, we use the notations ${\boldsymbol{u}}_d$, ${\theta}_d$, and ${\varphi}_{i,d}$, and we also denote their trace extensions by the same notations. Then, we assume the following conditions:
\begin{equation}\label{assump-2} 
	\begin{split}
		&{\boldsymbol{u}}_d \in L^4(0,T;\mathbf{H}^1(\Omega_b)), \, \partial_t{\boldsymbol{u}}_d\in L^2(0,T;(\mathbf{H}^1(\Omega_b))'),\quad{\text{such that } \operatorname{div} \boldsymbol{u}_d=0\, \,\text{ in }\,\Omega}\\
		& {\theta}_d \in L^2(0,T;H^1(\Omega_b)\cup H^1(\Omega_{ts})),\, \partial_t{\theta}_d \in L^2(0,T;(H^1(\Omega_b))'\cup (H^1(\Omega_{ts}))')\\
		&{\varphi}_{i,d} \in L^\infty(0,T;H^1(\Omega_i)\cap L^\infty(\Omega_i)),\qquad\forall i=b,ts.
\end{split}\end{equation}
The initial conditions of the temperatures and velocity on the boundary, ${\theta}_{d,0}:={\theta}_d(0,.)$ and ${\boldsymbol{u}}_{d,0}:={\boldsymbol{u}}_d(0,.)$, belong to ${L}^2(\Omega_b)\cup{L}^2(\Omega_{ts})$ and $\mathbf{L}^2(\Omega_b)$, respectively.
\\
Finally, the assumptions for the rest of the data are as follows:
\begin{equation}\label{assump-3}
	\boldsymbol{u}_{0}\in \mathbf{L}^2(\Omega_b),\quad \theta_{i,0}\in L^2(\Omega_i)\text{ for $i=b,ts$ }\quad \text{and}\quad 
	\boldsymbol{F}\in L^2(\Omega_{T,b}).
\end{equation}

Now we define what we mean by a weak solution of our system \eqref{S1}-\eqref{S3}. We also supply our main existing result.
\begin{definition}\label{defSol} We say that  $(\boldsymbol{u},\theta_b,\theta_{ts},\varphi_{b},\varphi_{ts})$ is a weak solution to System (\ref{S1}), (\ref{S2}) and (\ref{S3}), if 
	{
		\begin{equation*}
			\begin{split}
				& \boldsymbol{u}-\boldsymbol{u}_d \in L^\infty(0,T; \mathcal{L}_{\boldsymbol{b}})\cap L^2(0,T; \boldsymbol {\mathcal{H}}^{\boldsymbol{u}}_{\boldsymbol{0}}),\;\;\partial_t \boldsymbol{u} \in L^1(0,T;(\boldsymbol {\mathcal{H}}^{\boldsymbol{u}}_{\boldsymbol{0}})^\prime),\\
				&\theta_i-\theta_{d} \in L^\infty(0,T; L^2(\Omega_i))\cap L^2(0,T; \mathcal H^{\theta_{i}}_{0}),\;\;\partial_t \theta_i\in L^1(0,T;(\mathcal H^{\theta_{i}}_{0})^\prime), \quad \text{for } i=b,ts\\
				&\varphi_i-\varphi_{i,d} \in L^2(0,T; \mathcal H^{\varphi_{i}}_{0})\cap L^\infty(0,T;  L^\infty(\Omega_i))
				,\quad \text{for } i=b,ts
			\end{split}
		\end{equation*}
	} 
	and the following identities hold
	\begin{eqnarray}\label{model*}
		\begin{split}
			\int_{0}^{T}	\left\langle\partial_t\boldsymbol{u}, \psi\right\rangle\, dt+\int_{0}^{T}a_{u}(\theta_b ; \boldsymbol{u}, \boldsymbol{\psi})\, dt+\int_{0}^{T}b(\boldsymbol{u}, \boldsymbol{u}, \boldsymbol{\psi})\, dt
			-\int_{0}^{T}(\boldsymbol{F}, \boldsymbol{\psi})\, dt&=0,
			\\  
			\sum_{i=b,ts}\int_{0}^{T}\left\langle\partial_t\theta_{i}, S_{i} \right\rangle\, dt+\sum_{i=b,ts}\int_{0}^{T}a_{\theta_{i}}(\theta_{i} ; \theta_{i}, S_{i} )\, dt+\int_{0}^{T}d(\boldsymbol{u}, \theta_b, S_{b})\, dt
			-\sum_{i=b,ts}\int_{0}^{T}c_{\varphi_{i}}(\theta_{i} ,\varphi_{i},S_{i})\, dt 
			&=0,
			\\ 	  		
			\sum_{i=b,ts}	\int_{0}^{T}a_{\varphi_{i}}(\theta_{i} ; \varphi_{i}, \phi_i )&=0,
		\end{split}
	\end{eqnarray}   
	{
		for all test functions $\boldsymbol{\psi} \in D([0,T); \boldsymbol {\mathcal{H}}^{\boldsymbol{u}}_{\boldsymbol{0}})$, $S_i \in D([0,T);  {\mathcal{H}}^{\theta_i}_{\boldsymbol{0}})$ 
		and $\phi_i\in L^2(0,T; \mathcal H^{\varphi_{i}}_{0})$,
		for $i=b,ts$, with $S_{b}=S_{ts}$ on $\Sigma_{T,7}\cup \Sigma_{T,8}$ and
	}
	\begin{eqnarray}
		\boldsymbol{u}(0,\boldsymbol{x}) &=\boldsymbol{u}_{0}(\boldsymbol{x})  & \hbox { in }  \Omega\label{eqvar4}, \\
		\theta_i(0,\boldsymbol{x}) &=\theta_{i,0}(\boldsymbol{x})  & \hbox { in }  \Omega, \quad \text{ for } i=b,ts. \label{eqvar5}	\end{eqnarray}
	
\end{definition}

It is significant to highlight that variational formulations of the heat equations are not as straightforward as those of the Navier-Stokes system and the elliptic equations. While there is no difficulty in writing the convection term $d(\boldsymbol{u},\theta_b,\cdot)=\boldsymbol{u}.\nabla\theta_b$ in a variational sense, this term can be tested by a function in $\mathcal{H}^{\theta_b}_0\subset H^1\hookrightarrow L^{4}(\Omega_b)$. Since we seek the temperature $\theta_b$ in $\mathcal{H}^{\theta_i}_0$ and the velocity in ${\boldsymbol{\mathcal H}}^{\boldsymbol{u}}_{\boldsymbol{0}}\subset \boldsymbol{H}^1\hookrightarrow \boldsymbol{L}^{4}(\Omega_b)$.
However, the quadratic terms $\sigma_i(\theta_i)|\nabla \varphi_i|^{2}$, $i=b,ts$, cannot be tested by functions in $\mathcal{H}^{\theta_i}_0$ since we are looking for the potentials $\varphi_i$ in $\mathcal H^{\varphi_{i}}_{0}$, as this term belongs to $L^{1}(\Omega_i)$. There are different methods to choose suitable function spaces. 
For example, we can take the test functions in the space $\mathcal{H}^{\theta_i}_0\cap L^{\infty}(\Omega_i)$~\cite{bernardi2018finite, salhi2022well}. Alternatively, if we set the boundary condition of the potential $\varphi_{d}$ in the space $H^{1/2}(\Sigma_i)\cap L^{\infty}(\Sigma_i)$, we can choose the test functions for the heat equations in $\mathcal{H}^{\theta_i}_0$. A discussion of this approach can be found in~\cite{antontsev1994thermistor}.

\begin{theorem}\label{theo1}
	Assume conditions \eqref{assump-1}, \eqref{assump-2} and \eqref{assump-3}  hold. Then the RF-ablation model 
	(\ref{S1}), (\ref{S2}) and (\ref{S3}) possesses a weak solution in the sense of Definition \ref{defSol}.
\end{theorem}
The proof of Theorem \ref{theo1} is divided into a series of steps outlined in Section \ref{proof-thm1}. In Subsection \ref{approximation}, we construct an approximate solution and demonstrate its convergence. The convergence proof relies on several uniform a priori estimates and compactness arguments, which are established in Subsection \ref{estim-apro}. Finally, we conclude the proof of our result in Subsection \ref{conclude-proof}.


\section{Proof of Theorem \ref{theo1}}\label{proof-thm1}


\subsection{ Faedo-Galerkin approximate solutions}\label{approximation}

Let, for $i=b,ts$, $\left\{ (\xi_\ell, \alpha_{i,\ell}, \beta_{i,\ell}) \right\}_{\ell=1,2, \ldots} \subset \boldsymbol{\mathcal{H}}^{\boldsymbol{u}}_{0} \times \mathcal{H}^{\theta_{i}}_0 \times \mathcal{H}^{\varphi_{i}}_0$ be an orthonormal basis of $\mathcal{L}_{\boldsymbol{b}} \times \mathcal{H}^{\theta_{i}}_0 \times \mathcal{H}^{\varphi_{i}}_0$. We set ${\boldsymbol{V}}^{n} \times X^n_{i} \times Y^n_{i} = \operatorname{span} \left\{ (\xi_1, \alpha_{i,1}, \beta_{i,1}), \ldots, (\xi_n, \alpha_{i,n}, \beta_{i,n}) \right\}$.
The approximate Faedo-Galerkin problem to be solved is then:
Determine $\boldsymbol{u}^n-\boldsymbol{u}_d \in H^1(0,T; \boldsymbol{\mathcal{H}}^{\boldsymbol u}_{0})$, $\theta^n_i -\theta_d\in H^1(0,T; \mathcal{H}^{\theta_{i}}_0)$, and $\varphi^n_i -{\varphi}_{i,d}\in H^1(0,T; \mathcal{H}^{\varphi_{i}}_0)$,
\begin{equation}\label{probl-var-appr}
	\begin{split}
		\left\langle\partial_t\boldsymbol{u}^n, \xi_\ell \right\rangle+a_{\boldsymbol{u}}(\theta^n_{b}, \boldsymbol{u}^n, \xi_\ell)+b(\boldsymbol{u}^n, \boldsymbol{u}^n, \xi_\ell)-(\boldsymbol{F}^n, \xi_\ell)&=0, \\
		\sum_{i=b,ts} \left\langle\partial_t\theta^n_{i}, \alpha_{i,\ell} \right\rangle 
		+\sum_{i=b,ts}a_{\theta_{i}}(\theta^n_{i} ; \theta^n_{i}, \alpha_{i,\ell})
		+ d(\boldsymbol{u}^n, \theta^n_{b}, \alpha_{1,\ell})
		-\sum_{i=b,ts} c_{\varphi^n_{i}}(\theta^n_{i} ,\varphi^n_{i},\alpha_{i,\ell})&= 0,
		\\ 	  		
		\sum_{i=b,ts}a_{\varphi_{i}^n}(\theta_{i}^n,\varphi_i^n, \beta_{i,\ell} )&=0,
	\end{split}
\end{equation} 
for $\ell=1,\dots,n$, where 
\begin{equation}
	\begin{split}
		\boldsymbol{u}^{n}=\sum_{\ell=1}^{n} \boldsymbol{u}^n_{\ell} \xi_\ell, 
		\qquad \theta^n_j= \sum_{\ell=1}^{n} \theta_{i,\ell}^{n} \alpha_{i,\ell},
		\qquad \varphi^n_j= \sum_{\ell=1}^{n} \varphi_{i,\ell}^{n} \beta_{i,\ell}\quad \text{ for $i=b,\, ts$}.
	\end{split}
\end{equation}
The initial conditions of the ODE system are then given by
\begin{equation}\label{syst:ICreg}
	\begin{array}{lcl}
		\displaystyle \boldsymbol{u}^n(0)=\sum_{\ell=0}^{n}\boldsymbol{u}_{0,\ell}\xi_\ell,& &\textrm{where }\boldsymbol{u}_{0,\ell}=\left( \boldsymbol{u}_0,\xi_\ell\right)_{L^2}, \\
		\displaystyle \theta^n_i(0)=\sum_{\ell=0}^{n}  \theta_{i,0,\ell} \alpha_{i,\ell},&  &\textrm{where }\theta_{i,0,\ell}=\left( \theta_{i,0}, \alpha_{i,\ell}\right)_{L^2}, 
	\end{array}
\end{equation}
for $i=b,ts$. We use the following assumption for the initial conditions
\begin{equation}\label{eq:uie-init}
	\boldsymbol{u}_0 \in L^2(\Omega_b),\, \, \theta_{b,0}\in L^2(\Omega_b) \,\, \text{and}\,\,
	\theta_{ts,0}\in L^2(\Omega_{ts}).
\end{equation}
In \eqref{probl-var-appr} we have used a finite dimensional approximation of $\boldsymbol{F}$ 
$$
\boldsymbol{F}^n(t,x)=\sum_{\ell=0}^{n}\left( \boldsymbol{F},\xi_\ell\right)_{L^2}(t)\xi_\ell(x).
$$
Using the orthonormality of the basis, we can write \eqref{probl-var-appr} more explicitly as a system of ordinary differential equations:
\begin{equation}\label{probl-var-appr-ODE}
	\begin{array}{rll}
		\displaystyle 	\frac{d \boldsymbol{u}^n_{\ell}}{dt}&=& \displaystyle -a_{\boldsymbol{u}}(\theta^n_{b}, \boldsymbol{u}^n, \xi_\ell)-b(\boldsymbol{u}^n, \boldsymbol{u}^n, \xi_\ell)
		+(\boldsymbol{F}^n, \xi_\ell)  \\ 
		&:=& \displaystyle F_{\ell,\boldsymbol{u}^n}(t, \left\{\boldsymbol{u}^n_{k}\right\}_{k=1}^{n},\left\{\theta^n_{b,k}\right\}_{k=1}^{n}),\\ \\
		\displaystyle		\sum_{i=b,ts}\frac{d \theta_{i,\ell}^{n}}{dt}&=&\displaystyle - d(\boldsymbol{u}^n, \theta^n_{b}, \alpha_{1,\ell})
		-\sum_{i=b,ts}a_{\theta_{i}}(\theta^n_{i} ; \theta^n_{i}, \alpha_{i,\ell})
		+\sum_{i=b,ts} c_{\varphi^n_{i}}(\theta^n_{i} ,\varphi^n_{i},\alpha_{i,\ell})\\ 
		&:=&\displaystyle  \sum_{i=b,ts}F_{\ell,\theta_i}(t, \left\{\boldsymbol{u}^n_{k}\right\}_{k=1}^{n},
		\left\{\theta^n_{b,k}\right\}_{k=1}^{n},\left\{\theta^n_{ts,k}\right\}_{k=1}^{n},
		\left\{\varphi^n_{b,k}\right\}_{k=1}^{n}\left\{\varphi^n_{ts,k}\right\}_{k=1}^{n}),\\
		\\ 	  		
		\displaystyle		\sum_{i=b,ts}a_{\varphi_{i}^n}(\theta_{i}^n,\varphi_i^n, \beta_{i,\ell} )&=&0,
	\end{array}
\end{equation}  
for $i=b,ts$. 
From the assumptions on the data of the model, the functions $F_{\ell,\boldsymbol{u}^n}$ and $F_{\ell,\theta_i}$ (for $i=b,ts$) are Caratheodory functions. Therefore, according to  ordinary differential equation  theory, functions $\left\{\boldsymbol{u}^n_{k}\right\}_{k=1}^{n}$, $\left\{\theta^n_{b,k}\right\}_{k=1}^{n}$, and $\left\{\theta^n_{ts,k}\right\}_{k=1}^{n}$, satisfying the equations,  exist and are absolutely continuous. Consequently, a weak local solution exists for all $t\in (0,t_0)$ with $0<t_0<T$. To establish the existence of $\varphi_i^n$, we regularize the elliptic equation of $\varphi_i^n$ in a manner akin to \cite{B19}. Specifically, we add the term $\varepsilon_n \sum_{i=b,ts}(\partial_t \varphi_i^n, \beta_{i,\ell})$ to the left-hand side of the last equation in \eqref{probl-var-appr-ODE}, where $\varepsilon_n:=1/n$. We fix the initial data as $\varphi_{i}^n(0)=0$  and choosing the regularity condition of $\varphi_d$ that is $\partial_t \varphi_d=0$. Then the system involving the new unknown $\varphi_{i}^{n}$ (for $i=b,ts$) becomes a well-posed ordinary differential equation problem, and solutions are defined globally on $[0,T]$. The regularity parameter $\varepsilon_n$ tends towards 0 (as $n$ to $+\infty$) at the end of the calculations, so to simplicity we neglect these regularity terms.
	Then, that the Galerkin solutions $(\boldsymbol{u}^n,\pi^n,\theta_b^n,\theta_{ts}^n,\varphi_{b}^n,\varphi_{ts}^n)$
	satisfy the following weak formulation :

\begin{eqnarray}\label{wf-apriori}
	\begin{split}
		\int_{0}^{T}	\left\langle\partial_t\boldsymbol{u}^n, \psi\right\rangle\, dt+\int_{0}^{T}a_{\boldsymbol{u}}(\theta_b^n ; \boldsymbol{u}^n, \boldsymbol{\psi})\, dt+\int_{0}^{T}b(\boldsymbol{u}^n, \boldsymbol{u}^n, \boldsymbol{\psi})\, dt
		-\int_{0}^{T}(\boldsymbol{F}^n, \boldsymbol{\psi})\, dt&=0,
		\\  
		\sum_{i=b,ts}\int_{0}^{T}\left\langle\partial_t\theta_{i}^n, S_{i} \right\rangle\, dt+\sum_{i=b,ts}\int_{0}^{T}a_{\theta_{i}}(\theta_{i}^n ; \theta_{i}^n, S_{i} )\, dt+\int_{0}^{T}d(\boldsymbol{u}^n, \theta_b^n, S_{b})\, dt
		-\sum_{i=b,ts}\int_{0}^{T}c_{\varphi_{i}}(\theta_{i}^n ,\varphi_{i}^n,S_{i})\, dt 
		&=0,
		\\ 	  		
		\sum_{i=b,ts}	\int_{0}^{T}a_{\varphi_{i}}(\theta_{i}^n ; \varphi_{i}^n, \phi_i )&=0,
	\end{split}
\end{eqnarray}  
for all test functions $\boldsymbol{\psi} \in \boldsymbol D([0,T); \boldsymbol {\mathcal{H}}^{\boldsymbol{u}}_{\boldsymbol{0}})$,
	$S_i \in D([0,T); \mathcal{H}^{\theta_i}_0\cap L^{\infty}(\Omega_i))$ and $\phi_i\in D([0,T); \mathcal{H}^{\varphi_i}_0)$,
	for $i=b,ts$, with $S_{b}=S_{ts}$ on $\Sigma_{T,7}\cup \Sigma_{T,8}$.
	
Throughout the rest of the paper, we will always use positive constants $c$, $C(\cdot)$, $C(\cdot,\cdot)$, $C$, $C_1$, $C_2\cdots$,  which are not specified and which may differ from line to line.
\subsection{Basic a priori estimates}\label{estim-apro}
 To establish the global existence of the Faedo-Galerkin weak solution, we rely on a series of basic energy-type estimates as presented in the following lemma. 
\begin{lemma} 
	\label{lem2}
Under the above assumptions, there exists a constant $C>0$ not depending on $n$ such that
		\begin{eqnarray}
			\label{eq:Lemme1}
			\left\|\boldsymbol{u}^n\right\|_{L^{\infty}(0,T;L^{2}(\Omega_b))}+\left\|\nabla \boldsymbol{u}^n\right\|_{L^2(\Omega_{T,b},\mathbb{R}^3)}
			+\left\|\partial_t\boldsymbol{u}^{n}\right\|_{L^1(0,T;\boldsymbol{H}^{-1}(\Omega_{b}))}
			&\leq & C, 
			\\
			\label{eq:Lemme2}
			\sum_{i=b,ts}\left\|\varphi_{i}^n\right\|_{L^{\infty}(\Omega_{T,i})}
			+\sum_{i=b,ts} \left\|\nabla \varphi_{i}^n\right\|_{L^2(\Omega_{T,i})}
			&\leq & C,
			\\
			\label{eq:Lemme3}
			\sum_{i=b,ts}	\left\| \theta_i^n\right\|_{L^{\infty}(0,T;L^{2}(\Omega_i))}
			+\sum_{i=b,ts} \left\|\nabla \theta_i^n\right\|_{L^2(\Omega_{T,i})}
			+\sum_{i=b,ts}\left\|\partial_t\theta_i^{n}\right\|_{L^1(0,T;H^{-1}(\Omega_{i}))}
			&\leq & C.
		\end{eqnarray}
	
\end{lemma}

\begin{proof}
	{\it  Proof of (\ref{eq:Lemme1})}. First, we substitute $\boldsymbol{\psi}:=\boldsymbol{u}^n-\boldsymbol{u}_d$ into \eqref{wf-apriori}. Then, we disregard the time integration and combine the resulting equations to obtain:
	\begin{equation}
		\begin{split}\label{est1-apriori}
			&\frac{1}{2}\dfrac{d}{dt}\int_{\Omega_{b}}|\boldsymbol{u}^n-\boldsymbol{u}_d|^2\, dx
			+a_{u}(\theta^n_{b} ; \boldsymbol{u}^n-\boldsymbol{u}_d, \boldsymbol{u}^n-\boldsymbol{u}_d) \\
			=&-a_{u}(\theta^n_{b} ; \boldsymbol{u}_d, \boldsymbol{u}^n-\boldsymbol{u}_d)
			-b(\boldsymbol{u}^n, \boldsymbol{u}^n, \boldsymbol{u}^n-\boldsymbol{u}_d) 
			+\int_{\Omega_{b}}\boldsymbol{F}^n\cdot (\boldsymbol{u}^n- \boldsymbol{u}_d)\, dx-\int_{\Omega_{b}} \partial_t \boldsymbol{u}_d (\boldsymbol{u}^n-\boldsymbol{u}_d)\\
			:=&I_1+I_2+I_3+I_4.
		\end{split}
	\end{equation}
	From the definition of the form $a_u$,  it follows  that
	\begin{equation}\label{est1-apriori-1}
		\begin{split}
			|I_1|&=|-a_{u}(\theta^n_{b} ; \boldsymbol{u}_d, \boldsymbol{u}^n-\boldsymbol{u}_d)| 
			\\
			&\leq\overline{\nu}\|\nabla\boldsymbol{u}_d\|_{L^2} \|\nabla(\boldsymbol{u}^n-\boldsymbol{u}_d)\|_{L^2}
			\\
			&\leq
			\delta\|\boldsymbol{u}^n-\boldsymbol{u}_d\|_{H^1}^2+C(\delta,\overline{\nu})\|\nabla\boldsymbol{u}_d\|_{L^2}^2.
		\end{split}
\end{equation}
	Using the properties of the trilinear form $b$, $\operatorname{div}\boldsymbol{u}^n=0$ and $\boldsymbol{u}^n=\boldsymbol{u}_d$ on $\Sigma_b$, we obtain 
	\begin{equation*}\begin{split}
			-I_2&=
			b(\boldsymbol{u}^n, \boldsymbol{u}^n, \boldsymbol{u}^n-\boldsymbol{u}_d)\\
			&=\sum_{i,j=1}^N \int_{\Omega_b}\boldsymbol{u}^n_i\, \partial_{x_i} \boldsymbol{u}^n_j \, (\boldsymbol{u}^n_j-\boldsymbol{u}_{j,d})\, dx\\
			&=-\sum_{i,j=1}^N \int_{\Omega_b} \partial_{x_i} \boldsymbol{u}^n_i\,\boldsymbol{u}^n_j \, (\boldsymbol{u}^n_j-\boldsymbol{u}_{j,d})\, dx 
			-  \sum_{i,j=1}^N \int_{\Omega_b}  \boldsymbol{u}^n_i\,\boldsymbol{u}^n_j \, \partial_{x_i} (\boldsymbol{u}^n_j-\boldsymbol{u}_{j,d})\, dx\\
			&=-\int_{\Omega_b}\operatorname{div} \boldsymbol{u}^n \,\boldsymbol{u}^n (\boldsymbol{u}^n-\boldsymbol{u}_{d})\, dx
			-\frac{1}{2}\sum_{i,j=1}^N\int_{\Omega_b}\boldsymbol{u}^n_i \,\partial_{x_i} ({\boldsymbol{u}^n_j})^2\, dx
			+\sum_{i,j=1}^N \int_{\Omega_b}  \boldsymbol{u}^n_i\,\boldsymbol{u}^n_j \, \partial_{x_i} \boldsymbol{u}_{j,d}\, dx\\
			&=\frac{1}{2}\int_{\Omega_b}\operatorname{div} \boldsymbol{u}^n \,|\boldsymbol{u}^n|^2\, dx-\frac{1}{2}\int_{\Sigma_b} \left|\boldsymbol{u}_d\right|^2 \boldsymbol{u}_d \cdot \boldsymbol{n}_b dy
			+\sum_{i,j=1}^N \int_{\Omega_b}  \boldsymbol{u}^n_i\,\boldsymbol{u}^n_j \, \partial_{x_i} \boldsymbol{u}_{j,d}\, dx\\
			&=-\frac{1}{2}\int_{\Sigma_b} \left|\boldsymbol{u}_d\right|^2 \boldsymbol{u}_d \cdot \boldsymbol{n}_b dy
			+\sum_{i,j=1}^N \int_{\Omega_b}  \boldsymbol{u}^n_i\,\boldsymbol{u}^n_j \, \partial_{x_i} \boldsymbol{u}_{j,d}\, dx
			\\	&:=I_{21}+I_{22}.
	\end{split}\end{equation*} 
	Obviously, $I_{21}$ is bounded. Let us proceed with the estimation of the term $I_{22}$. To begin, we can reformulate $I_{22}$ as follows:
	\begin{equation*}\begin{split}
			I_{22}=	b(\boldsymbol{u}^n, \boldsymbol{u}_d, \boldsymbol{u}^n)
			&=b(\boldsymbol{u}_d, \boldsymbol{u}_d, \boldsymbol{u}_d)+b(\boldsymbol{u}^n-\boldsymbol{u}_d, \boldsymbol{u}_d, \boldsymbol{u}_d)+b(\boldsymbol{u}_d, \boldsymbol{u}_d, \boldsymbol{u}^n-\boldsymbol{u}_d)+b(\boldsymbol{u}^n-\boldsymbol{u}_d, \boldsymbol{u}_d, \boldsymbol{u}^n-\boldsymbol{u}_d).
	\end{split}\end{equation*}
	Since $\boldsymbol{u}_d$ is assumed to be sufficiently regular, we have the boundedness of $b(\boldsymbol{u}_d, \boldsymbol{u}_d, \boldsymbol{u}_d)$. 
	While  the second and the third terms on the right-hand side of $I_{22}$ can be estimated by
	\begin{equation*}\begin{split}
		|b(\boldsymbol{u}^n-\boldsymbol{u}_d, \boldsymbol{u}_d, \boldsymbol{u}_d)|+|b(\boldsymbol{u}_d, \boldsymbol{u}_d, \boldsymbol{u}^n-\boldsymbol{u}_d)|
		&\leq \|\boldsymbol{u}^n-\boldsymbol{u}_d\|_{L^4}\|\nabla\boldsymbol{u}_d\|_{L^2}\|\boldsymbol{u}_d\|_{L^{4}}+\|\boldsymbol{u}_d\|_{L^{4}}\|\nabla\boldsymbol{u}_d\|_{L^2}\|\boldsymbol{u}^n-\boldsymbol{u}_d\|_{L^4}
		\\&\leq c\|\boldsymbol{u}^n-\boldsymbol{u}_d\|_{H^1}\|\nabla\boldsymbol{u}_d\|_{L^2}\|\boldsymbol{u}_d\|_{L^{4}}+c\|\boldsymbol{u}_d\|_{L^{4}}\|\nabla\boldsymbol{u}_d\|_{L^2}\|\boldsymbol{u}^n-\boldsymbol{u}_d\|_{H^1}
		\\&\leq \delta\|\boldsymbol{u}^n-\boldsymbol{u}_d\|_{H^1}^2+ C(\delta)\|\nabla\boldsymbol{u}_d\|_{L^2}^2\|\boldsymbol{u}_d\|_{L^{4}}^2.
	\end{split}
\end{equation*}
	Furthermore, we can apply Young's inequality with a parameter $\delta$ ($ab \leq \delta a^{p} + C(\delta) b^{q}$, 
	where $a, b > 0$, $\delta > 0$, $1 < p, q < \infty$, and $1/p + 1/q = 1$, with $C(\delta) = (\delta p)^{-q/p} q^{-1}$)
	and employ the Gagliardo-Nirenberg interpolation inequality (cf.~\cite[Theorem 5.8]{Adams2003}), to derive the following:
	\[\begin{aligned}
		{\|\boldsymbol{u}^n - \boldsymbol{u}_d\|}_{\boldsymbol{L}^{4}(\Omega_b)} &\leq c{\|\boldsymbol{u}^n - \boldsymbol{u}_d\|}_{\boldsymbol{H}^1(\Omega_b)}^{\zeta} \|\boldsymbol{u}^n - \boldsymbol{u}_d\|_{\boldsymbol{L}^{2}(\Omega_b)}^{1-\zeta}, \hbox{ for } \zeta = N/4.
	\end{aligned}\]
	Thus, the remaining term in $I_{22}$ can be estimated as follows
		\begin{equation*}\begin{split}
			|b(\boldsymbol{u}^n-\boldsymbol{u}_d, \boldsymbol{u}_d, \boldsymbol{u}^n-\boldsymbol{u}_d)|
			&\leq \|\boldsymbol{u}^n-\boldsymbol{u}_d\|_{L^4}\|\nabla\boldsymbol{u}_d\|_{L^2}\|\boldsymbol{u}^n-\boldsymbol{u}_d\|_{L^4}
			\\&\leq c \|\boldsymbol{u}^n-\boldsymbol{u}_d\|_{L^2}^{2(1-\zeta)}\|\boldsymbol{u}^n-\boldsymbol{u}_d\|_{H^1}^{2\zeta}\|\nabla\boldsymbol{u}_d\|_{L^2} 
			\\&\leq \delta\|\boldsymbol{u}^n-\boldsymbol{u}_d\|_{H^1}^{2}+C(\delta)\|\nabla\boldsymbol{u}_d\|_{L^2} ^{\frac{1}{1-\zeta}}\|\boldsymbol{u}^n-\boldsymbol{u}_d\|_{L^2}^2.
		\end{split}
	\end{equation*}	
	Then, we deduce that 
		\begin{equation*}\begin{split}
			|I_{22}|&\leq \delta\|\boldsymbol{u}^n-\boldsymbol{u}_d\|_{H^1}^2+C(\delta)\|\nabla\boldsymbol{u}_d\|_{L^2} ^{\frac{1}{1-\zeta}}\|\boldsymbol{u}^n-\boldsymbol{u}_d\|_{L^2}^2.
		\end{split}
	\end{equation*}
	This implies
	\begin{equation}\label{est2-apriori-1}
		\left|I_2\right|
		\leq 
		\delta\|\boldsymbol{u}^n-\boldsymbol{u}_d\|_{H^1}^2+C(\delta,\overline{\nu})\|\nabla\boldsymbol{u}_d\|_{L^2}^2 +C(\delta)\|\nabla\boldsymbol{u}_d\|_{L^2} ^{\frac{1}{1-\zeta}}\|\boldsymbol{u}^n-\boldsymbol{u}_d\|_{\boldsymbol{L}^2}^2. 
	\end{equation}
	By applying the same arguments again, we obtain
	\begin{equation}\label{est3-apriori-1}
		\begin{split}
	\left|I_3\right|\leq c\|\boldsymbol{F}\|_{\boldsymbol{L}^2(\Omega_{b})}^2+c\|\boldsymbol{u}^n-\boldsymbol{u}_d\|_{\boldsymbol{L}^2(\Omega_{b})}^2
		\end{split}
	\end{equation}
	and  
	\begin{equation}\label{est4-apriori-1}
		\begin{split}
		\left|I_4\right|\leq \delta\|\boldsymbol{u}^n-\boldsymbol{u}_d\|_{\boldsymbol{L}^2(\Omega_{b})}^2 +C(\delta)\|\partial_t\boldsymbol{u}_d\|_{\boldsymbol{H}^{-1}(\Omega_{b})}^2.
			\end{split}
	\end{equation}
	Combining the results obtained in \eqref{est1-apriori-1}, \eqref{est2-apriori-1}, \eqref{est3-apriori-1}, and \eqref{est4-apriori-1}, and taking $\delta$ such that the constant $\displaystyle C(\underline{\nu},{\delta})=\frac{\underline{\nu}}{2}-\delta>0$, we can derive 
\begin{equation}\label{ineq:elast-m}
	\begin{split}
		&\dfrac{d}{dt}\|\boldsymbol{u}^n-\boldsymbol{u}_d\|^2_{\boldsymbol{L}^2(\Omega_{b})}+C(\underline{\nu},{\delta})\|\boldsymbol{u}^n-\boldsymbol{u}_d\|_{\boldsymbol{H}^1(\Omega_{b})}^2
		\\
		&\qquad\leq  
		C\|\boldsymbol{F}\|_{\boldsymbol{L}^2(\Omega_{b})}^2+C\|\partial_t\boldsymbol{u}_d\|_{\boldsymbol{H}^{-1}(\Omega_{b})}^2+
		C(\delta,\overline{\nu})\|\nabla\boldsymbol{u}_d\|_{L^2}^2+ 
		C\|\boldsymbol{u}^n-\boldsymbol{u}_d\|_{\boldsymbol{L}^2}^2\Bigl(1+ \|\nabla\boldsymbol{u}_d\|_{L^2} ^{\frac{1}{1-\zeta}}\Bigl).
	\end{split}
\end{equation}
	Since $\boldsymbol{u}_d\in \boldsymbol{L}^4(0,T;H^{1}(\Omega)),\,\partial_t\boldsymbol{u}_d\in \boldsymbol{L}^2(0,T;\boldsymbol{H}^{-1}(\Omega_{b}))$, $\boldsymbol{F}\in \boldsymbol{L}^2(\Omega_T)$,  $\boldsymbol{u}^n(0)$ and $\boldsymbol{u}_d(0) \in L^2(\Omega_b, \mathbb{R}^3)$, integrating \eqref{ineq:elast-m} over $(0, t)$ with $0 < t \leq T$, we can apply the Gröenwall inequality to obtain: 
	\begin{equation}\label{ineq:est-u_n-L2}
	\sup_{0 < t \leq T}	\|\boldsymbol{u}^n(t) - \boldsymbol{u}_d(t)\|^2_{L^2(\Omega_b, \mathbb{R}^3)} \leq C,
	\end{equation}
	 where $C > 0$ is a constant depending on the $L^2$ norm of $\boldsymbol{u}_d$, $\partial_t\boldsymbol{u}_d$, $\boldsymbol{F}$,  $\boldsymbol{u}_0$ and $\boldsymbol{u}_{d,0}$. Integrating again \eqref{ineq:elast-m} in time leads to the following inequality:
	\begin{equation}\label{ineq:est-u_m}
		\int_0^T \|\boldsymbol{u}^n(t)\|^2_{H^1(\Omega_b, \mathbb{R}^{3})} \, dt \leq C.
	\end{equation}
	
		Since the assumption \eqref{assump-1},  $\{\boldsymbol{u}^n\}_{n=1}^{\infty}$ is bounded in $L^{\infty}(0,T ;\boldsymbol{L}^2(\Omega_b))\cap L^{2}(0,T ;\boldsymbol{H}^1(\Omega_b))$ (cf. \eqref{ineq:est-u_n-L2} and \eqref{ineq:est-u_m}) and  $\boldsymbol{F}^n$ is bounded in $L^2(0,T;L^2(\Omega_b))$, then we have that $\{\partial_t\boldsymbol{u}^n\}_{n=1}^{\infty}$ is bounded in $L^{1}(0,T ; (\boldsymbol {\mathcal{H}}^{\boldsymbol{u}}_{\boldsymbol{0}})')$. Indeed, for all $\boldsymbol{w}\in\boldsymbol {\mathcal{H}}^{\boldsymbol{u}}_{\boldsymbol{0}}$ we have that
		\begin{equation*}
			\begin{split}
				|\left\langle\partial_t\boldsymbol{u}^n, \psi\right\rangle| &=|-a_{\boldsymbol{u}}(\theta_b^n ; \boldsymbol{u}^n, \boldsymbol{\psi})-b(\boldsymbol{u}^n, \boldsymbol{u}^n, \boldsymbol{\psi})+(\boldsymbol{F}^n, \boldsymbol{\psi})|
				\\&\leq|a_{\boldsymbol{u}}(\theta_b^n ; \boldsymbol{u}^n, \boldsymbol{\psi})| + |b(\boldsymbol{u}^n, \boldsymbol{u}^n, \boldsymbol{\psi})|+|(\boldsymbol{F}^n, \boldsymbol{\psi})|
				\\&\leq\bar{\nu}\|\nabla\boldsymbol{u}^n\|_{L^2}\|\nabla\boldsymbol{\psi}\|_{L^2}+c\|\boldsymbol{u}^n\|_{L^{2}}^{(1-\zeta)}\|\boldsymbol{u}^n\|_{H^1}^{\zeta+1}\|\boldsymbol{\psi}\|_{H^1}+\|\boldsymbol{F^n}\|_{L^2}\|\boldsymbol{\psi}\|_{L^2},\,\hbox{ where } \zeta = N/4
				\\&\leq C_1\|\nabla\boldsymbol{u}^n\|_{L^2}\|\boldsymbol{\psi}\|_{H^1}+C_2\|\boldsymbol{u}^n\|_{H^1}^{2}\|\boldsymbol{\psi}\|_{H^1}+C_3\|\boldsymbol{F^n}\|_{L^2}\|\boldsymbol{\psi}\|_{H^1},
			\end{split}
		\end{equation*}
		where $c_1$, $c_2$ and $c_3$ are positives constants. Further, (recall that $\boldsymbol{u}^n\in L^{2}(0,T ;\boldsymbol{H}^1(\Omega_b))$) 
		\begin{equation}
			\begin{split}
				\|\partial_t\boldsymbol{u}^n\|_{L^{1}(0,T;(\boldsymbol {\mathcal{H}}^{\boldsymbol{u}}_{\boldsymbol{0}})')}
				&\leq C_1T^{1/2}\left(\int_0^T\|\nabla\boldsymbol{u}^n\|_{L^2}^2dt\right)^{1/2}+C_2\int_0^T\|\boldsymbol{u}^n\|_{H^1}^{2}dt+C_3T^{1/2}\left(\int_0^T\|\boldsymbol{F}\|_{L^2}^2dt\right)^{1/2}
				\\&\leq C.
			\end{split}
		\end{equation}
		This achieves the estimate \eqref{eq:Lemme1}.
	
	
	\noindent \textit{Proof of (\ref{eq:Lemme2}).}
	By the maximum principle, we can deduce from the equation for $\varphi_i^n$ the following inequality:
	\begin{equation}\label{borne-de-fi}
		\underline{\varphi}_d \leq \varphi_{b}^n (t_1,x_1),\; \varphi_{ts}^n(t_2,x_2)\leq \overline{\varphi}_d\qquad \text{for all }(t_1,x_1)\in \Omega_{T,b},\,(t_2,x_2)\in \Omega_{T,ts},
	\end{equation}
	where $\underline{\varphi}_d=\displaystyle \min_{(t,x)\in \overline{\Omega}_{T,b}\cup \overline{\Omega}_{T,ts}}{\varphi}_d(t,x)$ and $\overline{\varphi}_d=\displaystyle \max_{(t,x)\in \overline{\Omega}_{T,b}\cup \overline{\Omega}_{T,ts}}{\varphi}_d(t,x)$. We define $\underline{\sigma}=\min\{\underline{\sigma}_{b},\underline{\sigma}_{ts}\}$ and $\overline{\sigma}=\max\{\overline{\sigma}_{b},\overline{\sigma}_{ts}\}$.
	Now, by substituting $\phi_i=\varphi_i^n-\varphi_d$ into \eqref{wf-apriori}, we can derive the following equation:
	\begin{equation}\label{estm-de-fi-H1}
		\begin{split}
			\underline{\sigma}\sum_{i=b,ts}\iint_{\Omega_{T,i}} \left|\nabla \varphi_i^n\right|^2\, dx\, dt
			&\leq \sum_{i=b,ts}\iint_{\Omega_{T,i}} \sigma_i(\theta_i^n) \nabla \varphi_i^n \cdot \nabla \varphi_i^n\, dx\, dt\\
			&=\sum_{i=b,ts}\iint_{\Omega_{T,i}} \sigma_i(\theta_i^n) \nabla \varphi_i^n \cdot \nabla \varphi_d\, dx\, dt\\
			&\leq \frac{\underline{\sigma}}{2} \sum_{i=b,ts}\iint_{\Omega_{T,i}} \left|\nabla \varphi_i^n\right|^2\, dx\, dt
			+C(\underline{\sigma}, \overline{\sigma})\sum_{i=b,ts}\iint_{\Omega_{T,i}} \left|\nabla \varphi_d\right|^2\, dx\, dt,
		\end{split}
	\end{equation}
	where we have used the assumptions \eqref{assump-1} and the constant $C(\underline{\sigma}, \overline{\sigma})>0$ is depending on $\underline{\sigma}$ and $ \overline{\sigma}$.
	This leads to the following conclusion for $i=b,ts$:
	\begin{equation}\label{norm-fi-H1}
		\sum_{i=b,ts}\iint_{\Omega_{T,i}} \left|\nabla \varphi_i^n\right|^2\, dx\, dt\leq C.
	\end{equation}

	
	\noindent {\it  Proof of (\ref{eq:Lemme3})}.
	Let us introduce $\underline{\eta}=\min\{\underline{\eta}_{b},\underline{\eta}_{ts}\}>0$ and $\overline{\eta}=\max\{\overline{\eta}_{b},\overline{\eta}_{ts}\}>0$. Now, we focus on \eqref{wf-apriori} without considering the time integration. We use the test function $S_i=\theta_i^n(t)-\theta_{d}(t)\in\mathcal{H}^{\theta_i}_0$. This allows us to derive the following inequality:
	
	\begin{equation}\label{inquality12}
		\begin{split}
			&\frac{1}{2}\sum_{i=b,ts} \frac{d}{dt}\|\theta_{i}^n(t)-\theta_d(t)\|_{L^2}^2+\underline{\eta}\sum_{i=b,ts}\int_{\Omega_i}|\nabla (\theta_i^n(t)-\theta_{d}(t))|^2\,dx
			\\
			&\leq
			\sum_{i=b,ts} \left(\frac{1}{2}\frac{d}{dt}\|\theta_{i}^n(t)-\theta_d(t)\|_{L^2}^2+a_{\theta_i}(\theta_{i}^n,\theta_{i}^n-\theta_d(t),\theta_i^n(t)-\theta_{d}(t))\right)\\
			&=-d\left(\boldsymbol{u}^n(t), \theta_b^n(t), \theta_b^n(t)-\theta_{d}(t)\right)
			-\sum_{i=b,ts}a_{\theta_i}(\theta_{i}^n,\theta_d(t),\theta_i^n(t)-\theta_{d}(t))
			+\sum_{i=b,ts}c_{\varphi_{i}}(\theta_{i}^n ; \varphi_{i}^n, \theta_i^n(t)-\theta_{d}(t))\\
			& \qquad
			-\sum_{i=b,ts}\int_{\Omega_i}\partial_t\theta_d(t)(\theta_i^n(t)-\theta_{d}(t))\,dx
			\\
			&:=J_1+J_2+J_3 +J_4.
		\end{split}
	\end{equation}
	We reformulate the first term in \eqref{inquality12} as follows: (Recall that $\operatorname{div}\boldsymbol{u}^n=0$ and $\theta_b=\theta_s$ on $\Sigma_8$ )
	\begin{equation}\label{equ241}
		\begin{aligned}
			-J_1 &= d\left(\boldsymbol{u}^n(t), \theta_b^n(t), \theta_b^n(t)-\theta_{d}(t)\right)\\
			&= \int_{\Omega_{b}}\boldsymbol{u}^n\cdot\nabla(\theta_b^n-\theta_{d}(t))(\theta_b^n(t)-\theta_{d}(t))\,dx+\int_{\Omega_{b}}\boldsymbol{u}^n\cdot\nabla\theta_{d}(t)(\theta_b^n(t)-\theta_{d}(t))\,dx\\
			&=-\frac{1}{2}\int_{\Omega_{b}}\operatorname{div}\boldsymbol{u}^n\,(\theta_b^n-\theta_{d}(t))^2\,dx+\frac{1}{2}\int_{\Sigma_{b}}\boldsymbol{u}^n\cdot\boldsymbol{n}_b\,(\theta_b^n-\theta_{d}(t))^2\,dy+\int_{\Omega_{b}}\boldsymbol{u}^n\cdot\nabla\theta_{d}(t)(\theta_b^n(t)-\theta_{d}(t))\,dx\\
			&=\int_{\Omega_{b}}\boldsymbol{u}^n\cdot\nabla\theta_{d}(t)(\theta_b^n(t)-\theta_{d}(t))\,dx.
		\end{aligned}
	\end{equation} 
Now, to get the estimate of $J_1$, we applied the Sobolev embedding $H^1 \hookrightarrow L^p$ with $1\leq p\leq 6$, used Hölder's inequality and Young's inequality with parameter $\delta$; then, we obtained the following estimation:
	\begin{equation}\label{estm-J1}
		\begin{aligned}
			\left|J_{1}\right| &=\left|-d\left(\boldsymbol{u}^n(t), \theta_d(t), \theta_b^n(t)-\theta_{d}(t)\right)\right|\\
			&\leq\|\boldsymbol{u}^n(t)\|_{\mathbf{L}^{4}}\left\| \theta_b^n(t)-\theta_{d}(t)\right\|_{\mathbf{L}^{4}}\left\|\nabla\theta_d(t)\right\|_{L^{2}}\\
			& \leq c \|\boldsymbol{u}^n(t)\|_{\mathbf{H}^{1}}\left\| \theta_b^n(t)-\theta_{d}(t)\right\|_{\mathbf{H}^{1}}\left\|\nabla\theta_d(t)\right\|_{L^{2}}
			\\
			& \leq \frac{\delta}{2}\left\|\theta_b^n(t)-\theta_{d}(t)\right\|_{H^{1}}^{2}+C(\delta)\|\boldsymbol{u}^n(t)\|_{\mathbf{H}^{1}}^2\left\|\nabla\theta_{d}(t)\right\|_{\mathbf{L}^{2}}^2.
		\end{aligned}
	\end{equation} 
	From the definition of the form $a_{\theta_i}$, we have 
	\begin{equation*}\begin{split}
			\left|a_{\theta_i}(\theta_{i}^n,\theta_d(t),\theta_i^n(t)-\theta_{d}(t))\right|&=\left|\int_{\Omega_i}\eta_{i}(\theta_{i}^n)\nabla\theta_d(t)\nabla(\theta_{i}^n-\theta_d(t))\, dx\right|
			\\
			&\leq\overline{\eta}\|\nabla\theta_d(t)\|_{\mathbf{L}^2(\Omega_i)}\|\nabla( \theta_i^n(t)-\theta_{d}(t))\|_{\mathbf{L}^2(\Omega_i)}
			\\
			&\leq\delta\|\theta_i^n(t)-\theta_{d}(t)\|_{{H}^1(\Omega_i)}^2+C(\delta)\|\nabla\theta_d(t)\|_{\mathbf{L}^2(\Omega_i)}^2,
	\end{split}\end{equation*}
	for $i=b,\,ts$ where $C(\delta)>0$.  
	This implies that
	\begin{equation}\label{est3-apriori}
		\begin{split}
			\left|J_2\right| &\leq \delta\sum_{i=b,ts}\|\theta_i^n(t)-\theta_{d}(t)\|_{{H}^1(\Omega_i)}^2+C(\delta)\sum_{i=b,ts}\|\nabla\theta_d(t)\|_{\mathbf{L}^2(\Omega_i)}^2.
		\end{split}
	\end{equation}
	According to the definition of $c_{\varphi}$ and using Green's formula, we obtain the following expressions:
	\begin{equation}\label{}
		\begin{aligned}
			J_3&=
			\sum_{i=b,ts}	c_{\varphi_{i}^n}(\theta_i^n ; \varphi_{i}^n, \theta_i^n-\theta_{d})\\
			&=\sum_{i=b,ts}\int_{\Omega_i} \sigma_i(\theta_i^n) \nabla \varphi_i \cdot \nabla \varphi_{i}^n (\theta_i^n-\theta_{d}) \, dx\\
			&=\sum_{i=b,ts}\int_{\Omega_i} \sigma_i(\theta_i^n) \nabla \varphi_{i}^n \cdot \nabla (\varphi_{i}^n (\theta_i^n-\theta_{d})) \, dx-\sum_{i=b,ts}\int_{\Omega_i} \sigma_i(\theta_i^n)  \varphi_{i}^n  \nabla \varphi_{i}^n\cdot \nabla (\theta_i^n-\theta_{d}) \, dx
			\\
			&=\sum_{i=b,ts}\int_{\Sigma_i} (\sigma_i(\theta_i^n) \nabla \varphi_{i}^n)\cdot \boldsymbol{n}_i \varphi_{i}^n (\theta_i^n-\theta_{d}) \, dy-\sum_{i=b,ts}\int_{\Omega_i} \nabla\cdot(\sigma_i(\theta_i^n) \nabla \varphi_{i}^n) \varphi_{i}^n (\theta_i^n-\theta_{d}) \, dx
			\\&\qquad\qquad-\sum_{i=b,ts}\int_{\Omega_i} \sigma_i(\theta_i^n)  \varphi_{i}^n  \nabla \varphi_{i}^n\cdot \nabla( \theta_i^n-\theta_{d}) \, dx\\
			&=-\sum_{i=b,ts}\int_{\Omega_i} \sigma_i(\theta_i)\varphi_{i}^n \nabla \varphi_{i}^n \cdot \nabla (\theta_i^n-\theta_{d})\, dx.
		\end{aligned}
	\end{equation}
	Then, by \eqref{borne-de-fi} and \eqref{norm-fi-H1},  we conclude that: 
	\begin{equation}\label{estem-hn}
		\begin{aligned}
			\left|J_3\right|=&\left|-\sum_{i=b,ts}\int_{\Omega_i} \sigma_i(\theta_i^n)\varphi_i^n \nabla \varphi_i^n \cdot \nabla (\theta_i^n-\theta_{d})\, dx\right|\\
			\leq &\overline{\sigma}\sum_{i=b,ts} ||\varphi_i^n||_{L^{\infty}} ||\nabla \varphi_i^n||_{L^2}||\nabla \theta_i^n-\theta_{d}||_{L^2}\\
			\leq & C \sum_{i=b,ts} ||\nabla \varphi_i^n ||_{\boldsymbol{L}^2}||\theta_i^n-\theta_{d}||_{{H}^1}
			\\
			\leq &\delta\sum_{i=b,ts}||\theta_i^n-\theta_{d}||_{{H}^1}^2+C(\delta)\sum_{i=b,ts}  ||\nabla \varphi_i^n ||^2_{\boldsymbol{L}^2}.
		\end{aligned}
	\end{equation}
	Similarly to $J_3$, we have for $J_4$
		\begin{equation}\label{estem-J4}
		\left|J_4\right|\leq \delta\sum_{i=b,ts}||\theta_i^n-\theta_{d}||_{{H}^1}^2+C(\delta)\sum_{i=b,ts}{\left\|\partial_t\theta_d(t)\right\|}_{H^{-1}(\Omega_i)}^{2}.
	\end{equation}
	Now, choosing $\delta$ sufficiently small, using the estimates \eqref{ineq:est-u_m}, \eqref{norm-fi-H1}, \eqref{estm-J1},  \eqref{est3-apriori}, \eqref{estem-hn} and \eqref{estem-J4}  that leads to the following estimate  
		\begin{equation}\label{equsmall}
	\begin{split}
	&	\frac{d}{d t}\left( \sum_{i=b,ts} {\left\|\theta_{i}^n(t)-\theta_d(t)\right\|}_{L^{2}(\Omega_i)}^{2}\right)+C_1\sum_{i=b,ts}||\theta_{i}^n(t)-\theta_d(t)||_{{H}^1}^2
	\\
	&\qquad	\leq C_2\left(\sum_{i=b,ts}{\left\|\partial_t\theta_d(t)\right\|}_{H^{-1}(\Omega_i)}^{2}
	+\sum_{i=b,ts}  ||\nabla \theta_d(t) ||^2_{\boldsymbol{L}^2} +\sum_{i=b,ts}  ||\nabla \varphi_i^n ||^2_{\boldsymbol{L}^2}
	+\|\boldsymbol{u}^n(t)\|_{\mathbf{H}^{1}}^2\left\|\nabla\theta_{d}(t)\right\|_{\mathbf{L}^{2}}^2
\right)	\\&\qquad \qquad+C_3\sum_{i=b,ts}{\left\|\theta_{i}^n(t)-\theta_d(t)\right\|}_{L^{2}(\Omega_i)}^{2} 
	\end{split}
	\end{equation}
	for some constants $C_1>0$, $C_2>0$  and $C_3>0$.
	Integrating \eqref{equsmall} over $(0, t)$ with $0 < t \leq T$, we get that
		\begin{equation}\label{equsmalsl}
			\begin{split}
				& \sum_{i=b,ts} {\left\|\theta_{i}^n(t)-\theta_d(t)\right\|}_{L^{2}(\Omega_i)}^{2}
				\leq\sum_{i=b,ts} {\left\|\theta_{i}^n(0)-\theta_d(0)\right\|}_{L^{2}}^{2}+ \alpha(t)+C_3\int_{0}^{t}\sum_{i=b,ts}{\left\|\theta_{i}^n(s)-\theta_d(s)\right\|}_{L^{2}(\Omega_i)}^{2}\,ds, 
		\end{split}
	\end{equation} 
	where $$\alpha(t)=C_2\int_{0}^{t}\left(\sum_{i=b,ts}{\left\|\partial_t\theta_d(s)\right\|}_{H^{-1}(\Omega_i)}^{2}
	+\sum_{i=b,ts}  ||\nabla \theta_d(s) ||^2_{\boldsymbol{L}^2} +\sum_{i=b,ts}  ||\nabla \varphi_i^n(s) ||^2_{\boldsymbol{L}^2}
	+\|\boldsymbol{u}^n(s)\|_{\mathbf{H}^{1}}^2\left\|\nabla\theta_{d}(s)\right\|_{\mathbf{L}^{2}}^2
	\right)\,ds	>0$$ 
	 The Gronwall's inequality implies
	\begin{equation}\label{equ26}
		\begin{split}
		 \sum_{i=b,ts}\sup_{0 < t \leq T}\left\|\theta_{i}^n(t)-\theta_d\right\|_{L^{2}(\Omega_i)}^{2} & \displaystyle \leq  C
		\end{split} 
	\end{equation}
	for some constant $C>0$. By the estimates \eqref{equsmall} and \eqref{equ26}, we can conclude that there exist constants $C_1>0$ and $C_2>0$ such that
	\begin{eqnarray}
		\sum_{i=b,ts}\left\|\theta^{n}_i(t)\right\|_{L^{\infty}\left(0,T ; L^{2}(\Omega_i)\right)} \leq C_1,  \label{equ27}
		\\
		\sum_{i=b,ts}\left\|\theta^{n}_i(t)\right\|_{L^{2}\left(0,T ; H^{1}(\Omega_i)\right)} \leq C_2.\label{equ28}   
	\end{eqnarray} 
	Now, since the assumption \eqref{assump-1} be satisfied and 
		from \eqref{equ27} and \eqref{equ28}, we deduce that $\{\theta^{n}_i\}_{n=1}^{\infty}$ is bounded in $L^{\infty}(0,T ;L^2(\Omega_i))\cap L^{2}(0,T ;{H}^1(\Omega_i))$, for all $i=b,ts$. Thus,  we collect the previous results on $\boldsymbol{u}^n$ and $\varphi_i^n$ to obtain that $a_{\theta_{i}}(\theta_{i}^n ; \theta_{i}^n, \cdot )$, $d(\boldsymbol{u}^n, \theta_b^n, \cdot)$ and $c_{\varphi_{i}}(\theta_{i}^n ,\varphi_{i}^n,\cdot)$ are bounded in $L^{2}(0,T ; H^{-1}(\Omega_i))$, $L^{1}(0,T ; H^{-1}(\Omega_b))$ and $L^{2}(0,T ; H^{-1}(\Omega_i))$, respectively. Consequently, $\{\partial_t\theta_i^{n}\}_{n=1}^{\infty}$ is bounded in $L^{1}(0,T ; H^{-1}(\Omega_i))$ for all $i=b,ts$. Indeed
		\begin{equation*}
			\begin{split} 
				\sum_{i=b,ts}\int_{0}^{T}\|\partial_t\theta_{i}^n\|_{H^{-1}(\Omega_i)}\, dt&\leq\sum_{i=b,ts}\int_{0}^{T}\|a_{\theta_{i}}(\theta_{i}^n ; \theta_{i}^n, \cdot)\|_{H^{-1}(\Omega_i)}\, dt+\int_{0}^{T}\|d(\boldsymbol{u}^n, \theta_b^n,.)\|_{H^{-1}(\Omega_b)}\, dt\\
				&\qquad+\sum_{i=b,ts}\int_{0}^{T}\|c_{\varphi_{i}}(\theta_{i}^n ,\varphi_{i}^n,\cdot)\|_{H^{-1}(\Omega_i)}\, dt
				\\&\leq C(\overline{\eta},T)\sum_{i=b,ts}\|\nabla\theta_i^n\|_{L^2(0,T;L^2)}+c_1\|\boldsymbol{u}^n\|_{L^2(0,T;H^1(\Omega_b))}\|\nabla\theta_b^n\|_{L^2(0,T;L^2(\Omega_b))}\\&\qquad+C(\overline{\sigma},T)\sum_{i=b,ts}\|\varphi_i^n\|_{L^{\infty}(0,T;L^{\infty}(\Omega_i))}\|\nabla\varphi_i^n\|_{L^2(0,T;L^2(\Omega_i))}
				\\&\leq C,
			\end{split}
		\end{equation*}
for some constant $C>0$.
	
\end{proof}

\subsection{Passage to the limit and concluding the proof of Theorem \ref{theo1}}\label{conclude-proof}
Thanks to Lemma \ref{lem2}, there exist subsequences of $\{\boldsymbol{u}^n\}_{n=1}^{\infty}$, $\{\varphi_i^{n}\}_{n=1}^{\infty}$, and $\{\theta_i^{n}\}_{n=1}^{\infty}$, which will be still denoted (for simplicity) as $\{\boldsymbol{u}^n\}_{n=1}^{\infty}$, $\{\varphi_i^{n}\}_{n=1}^{\infty}$, and $\{\theta_i^{n}\}_{n=1}^{\infty}$(for $i=b,\, ts$ ), respectively, such that:
{ 
	\begin{eqnarray}\label{equ1-limit-u}
		\left\{
		\begin{split}
			\boldsymbol{u}^{n} & \rightharpoonup \, \boldsymbol{u} \text { weakly in } L^{2}\left(0, T ;\boldsymbol{\mathcal H}^{\boldsymbol u}_{\boldsymbol 0}\right),\\
			\partial_t\boldsymbol{u}^{n} & \rightharpoonup \, \partial_t\boldsymbol{u}\text { weakly in } L^{1}\left(0, T ; \mathbf (\boldsymbol{\mathcal H}^{\boldsymbol u}_{\boldsymbol 0})^{\prime}\right),
			\\
			\nabla \boldsymbol{u}^{n} & \rightharpoonup \, \nabla \boldsymbol{u} \text { weakly in } L^{2}\left(0, T ;\boldsymbol{L}^2\right),
		\end{split}\right.
	\end{eqnarray}
	\begin{eqnarray}	\label{fi3}
		\left\{
		\begin{split}
			\varphi_i^n & \rightharpoonup \, \varphi_i\hbox { weakly in } L^{\infty}\left(0, T ;L^{\infty}(\Omega_i)\right),
			\\
			\varphi_i^n & \rightharpoonup \, \varphi_i\hbox { weakly in } L^{2}\left(0, T ;H^{1}(\Omega_i)\right),
			\\
			\varphi_i^n & \rightarrow \, \varphi_i \hbox { strongly in  } L^{2}\left(0, T ; L^{2}(\Omega_i)\right),
			\\
			\nabla \varphi_i^n & \rightharpoonup\, \nabla \varphi_i \hbox { weakly in } L^{2}\left(I ; \boldsymbol{L}^{2}(\Omega_i)\right),
		\end{split}
		\right.
	\end{eqnarray}
	\begin{eqnarray}
		\left\{
		\begin{split}
			\theta_i^{n} & \rightharpoonup \, \theta_i \text { weakly in } L^{2}\left(0, T ; H^{1}\right),\label{equ29}
			\\
			\partial_t\theta_i^{n} & \rightharpoonup \, \partial_t\theta_i \text { weakly in } L^{1}\left(0, T ; H^{-1}\right),\label{equ30}
			\\
			\nabla 	\theta_i^{n} & \rightharpoonup \, \nabla \theta_i \text { weakly in } L^{2}\left(0, T ; \textbf{L}^{2}\right).
		\end{split}
		\right.  	
	\end{eqnarray}
	By the Aubin–Lions–Simon compactness theorem ($\{u\in L^2(0,T,H^1), \partial_t u\in L^1(0,T,H^{-1})\}\hookrightarrow\hookrightarrow L^2(0,T,L^2)$, see also \cite[Theorem II.5.16]{boyer2012mathematical}), we get for $i=b,\,ts$:
	\begin{eqnarray} \label{limit-Aubin–Lions–Simon}
		\left\{
		\begin{split}
			\boldsymbol{u}^{n} & \rightarrow \, \boldsymbol{u} \text { strongly in } L^{2}\left(0, T ; \mathbf L^{2}\right),
			\\
			\boldsymbol{u}^{n} & \rightarrow \, \boldsymbol{u} \text { almost everywhere in } \Omega_{b,T},\\
			\theta_i^{n} & \rightarrow \, \theta_i \text { strongly in } L^{2}\left(0, T ; L^{2}\right),
			\\
			\theta_i^{n} & \rightarrow \, \theta_i \text { almost everywhere in } \Omega_{T} \text {.}
		\end{split}
		\right.
	\end{eqnarray}
}\\ 
Next, denote the differences $\tilde{\theta}_i^{n}=\theta_i^{n}-\theta_i$, $\boldsymbol{\tilde{u}}^{n}=\boldsymbol{u}^n-\boldsymbol{u}$, and $\tilde{\varphi}_i^n=\varphi_i^n-\varphi_i$ for $i=b,\,ts$. Then, by choosing the test functions $(\boldsymbol{\psi},S_i,\phi_i)\in\boldsymbol{\mathcal{H}}^{\boldsymbol{u}}_{\textbf{0}}\times \mathcal{H}^{\theta_{i}}_0\times \mathcal{H}^{\varphi_{i}}_0$ in the weak formulation \eqref{wf-apriori} and using the results \eqref{equ1-limit-u}-\eqref{limit-Aubin–Lions–Simon}, as well as the continuity of $\nu$, $\eta_{i}$, and $\sigma_{i}$, we obtain the following limits:
{
	\begin{equation}\label{limit0}
		\begin{split}
			a_{\boldsymbol{u}}(\theta_b^n ; \boldsymbol{u}^n, \boldsymbol{\psi})-a_{\boldsymbol{u}}(\theta_b ; \boldsymbol{u}, \boldsymbol{\psi})&=a_{\boldsymbol{u}}(\theta_b^n ; \boldsymbol{\tilde{u}}^{n}, \boldsymbol{\psi})+\int_{\Omega_b} \left[\nu(\theta_b^n)-\nu(\theta_b)\right] \mathbb{D}(\boldsymbol{u}):\nabla\boldsymbol{\psi} \,dx\\
			&\leq \bar{\nu}\int_{\Omega_b}\mathbb{D}(\boldsymbol{\tilde{u}^n}):\nabla\boldsymbol{\psi}\,dx+\int_{\Omega_b} \left[\nu(\theta_b^n)-\nu(\theta_b)\right] \mathbb{D}(\boldsymbol{u}):\nabla\boldsymbol{\psi} \,dx.
		\end{split} 
	\end{equation}
	Based on the weak convergence of $\nabla\boldsymbol{u}^n$ to $\nabla\boldsymbol{u}$ in $\boldsymbol L^2(\Omega_b)$, the first integral goes to $0$ when $n$ goes to $\infty$. Furthermore, we infer the second quantity's convergence to $0$ using the Lebesgue-dominated convergence (see, for example, \cite{Brezis1983}). Then, we get  
	\begin{equation}\label{limit1}
		a_{\boldsymbol{u}}(\theta_b^n ; \boldsymbol{u}^n, \boldsymbol{\psi})\xrightarrow{n\longrightarrow \infty}   a_{\boldsymbol{u}}(\theta_b ; \boldsymbol{u}, \boldsymbol{\psi}).
	\end{equation} 
	Furthermore, following the sames methods that proved \eqref{limit1}, we also show that : 
	\begin{equation}\label{limit3}
		\begin{split}
			a_{\varphi}(\theta_i^n ,\varphi_i^n,\phi_i)-
			a_{\varphi}({\theta}_i ,\varphi_i,\phi_i)&=a_{\varphi}(\theta_i ,\tilde{\varphi}_i^n,\phi_i)+\int_{\Omega_i}[{\sigma}_i({\theta}_i^n)-\sigma_i( {\theta}_i)] \nabla {\varphi}_i \cdot \nabla \phi_i \mathrm{~d}x
			&\xrightarrow{n\longrightarrow \infty} 0 ,
			\\
			a_{\theta}\left( {\theta}_i^n , \theta_i^n, S_i\right)-a_{\theta}\left( {\theta}_i ; \theta_i, S_i\right)&=a_{\theta}\left({\theta}_i^n , \tilde{\theta}_i^{n}, S_i\right)+\int_{\Omega_i}\left[\eta_i\left({\theta}_i^{n}\right)-\eta_i\left( {\theta}_{i}\right)\right] \nabla \theta_{i} \cdot \nabla S_i \, \mathrm{d}x &\xrightarrow{n\longrightarrow \infty} 0. 
		\end{split}
	\end{equation}	
}
{
	From the compact embedding $\boldsymbol H^1(\Omega_b)\hookrightarrow\hookrightarrow \boldsymbol L^q(\Omega_b)$ (resp.  $H^1(\Omega_i)\hookrightarrow\hookrightarrow  L^q(\Omega_i)$, for $i=b,\,ts$), with $1\leq q\leq 6$(for $N=3$).
	%
	%
	Then,  using again Lebesgue-dominated convergence, we have the followings strong convergences:
	\begin{eqnarray}
		b(\boldsymbol{u}^n, \boldsymbol{u}^n, \boldsymbol{\psi})&\xrightarrow{n\longrightarrow \infty} & b(\boldsymbol{u}, \boldsymbol{u}, \boldsymbol{\psi}),
		\label{limit2}
		\\
		d\left(\boldsymbol{u}^{n}, \theta_b^n, S_{b}\right)&\xrightarrow{n\longrightarrow \infty}&  d\left(\boldsymbol{u}, \theta_b, S_{b}\right).
		\label{limit5}
	\end{eqnarray}
}
Finally, using the results of \eqref{fi3} and continuity of $\sigma_{i}$, we obtain
\begin{equation}
	\begin{aligned}
		c_{\varphi}( {\theta}_i^n ,{\varphi}_i^n,S_i )-c_{\varphi}({\theta}_i ,{\varphi}_i,S_i )&=\int_{\Omega_i}  \sigma_i( {\theta}_i^n)  \varphi_i^n \cdot \nabla \varphi_i^n\cdot\nabla S_i\, dx-\int_{\Omega}\sigma_i({\theta}_i) \varphi_i \cdot \nabla \varphi_i\nabla S_i    \mathrm{~d}x
		\\&=\left(\int_{\Omega_i}  \sigma_i( {\theta}_i^n)  \tilde{\varphi}_i^n \cdot \nabla \varphi_i^n\cdot\nabla S_i\, dx+\int_{\Omega}\left[\sigma_i( {\theta}_i^n)-\sigma_i({\theta}_i)\right] \varphi_i  \nabla \varphi_i^n\cdot\nabla S_i\right.\\& \qquad\qquad 
		\left.+\int_{\Omega_i}  \sigma_i( {\theta}_i)  \tilde{\varphi}_i^n \cdot \nabla \varphi_i^n\cdot\nabla S_i\, dx \right)\xrightarrow{n\longrightarrow \infty} 0 .
	\end{aligned}
	\label{limit6}
\end{equation}
Now, by exploiting the results \eqref{limit1}-\eqref{limit6} we can pass immediately to the limit in the weak approximate formulation \eqref{probl-var-appr} as $n\longrightarrow \infty$. The result is: the triple $(\boldsymbol{u},\theta_i,\varphi_i)$ satisfies the following variational formulation: 
\begin{equation}\label{equ33} 
	\begin{aligned}
		\left\langle\partial_t\boldsymbol{u}, \boldsymbol{\psi}\right\rangle+a_{\boldsymbol{u}}( {\theta}_b; \boldsymbol{u}, \boldsymbol{\psi})+b(\boldsymbol{u}, \boldsymbol{u}, \boldsymbol{\psi})  -(\boldsymbol{F}, \boldsymbol{\psi})&=  0,\\
		\sum_{i=b,ts}\left(\left\langle\partial_t\theta_{i}, S_{i} \right\rangle+a_{\theta_{i}}(\theta_{i} ; \theta_{i}, S_{i} )-c_{\varphi_{i}}(\theta_{i} ,\varphi_{i},S_{i}) 
		\right)+d(\boldsymbol{u}, \theta_b, S_{b})&=0,
		\\ 	  		
		\sum_{i=b,ts}a_{\varphi_{i}}(\theta_{i} ; \varphi_{i}, \phi_i )&=0,
	\end{aligned} 
\end{equation}
for every $(\boldsymbol{\psi},S_i,\phi_i)  \in\boldsymbol {\mathcal{H}}^{\boldsymbol{u}}_{\boldsymbol{0}}\times \mathcal{H}^{\theta_i}_0\times \mathcal{H}^{\varphi_i}_0$ with $S_b=S_{ts}$ on $\Sigma_b^N$ and almost every $t \in I$ and the initial condition
\begin{eqnarray}\label{equ34}
	\boldsymbol{u}(\boldsymbol{x}, 0)& = \boldsymbol{u}_{0}(\boldsymbol{x}),&\\
	\theta_i(\boldsymbol{x}, 0)&=\theta_i^{0}(\boldsymbol{x}), & \text { in } \Omega. \text{ for all }i=b,ts.
\end{eqnarray}

{
	To complete the proof of Theorem \ref{theo1}, we also need to demonstrate that $\partial_t \boldsymbol{u} \in \mathrm{L}^1\left(0, T; (\mathcal{H}^{\boldsymbol{u}}_0)^{\prime}\right)$ and $\partial_t\theta_i\in L^{1}(0,T ; (\mathcal{H}^{\theta_{i}})^{\prime}$ for all $i=b,ts$.
	Indeed, the first equation of \eqref{equ33} can be expressed as follows:
	\begin{equation}\label{eq:u-derivative}
		\frac{\mathrm{d}}{\mathrm{d} t}\langle\boldsymbol{u}, \psi\rangle=\langle\operatorname{div}\left(\nu(\theta_b) \mathbb{D}(\boldsymbol{u})\right)-B(\boldsymbol{u})+\boldsymbol{F}, \boldsymbol{\psi}\rangle \text { for all } \boldsymbol{\psi} \in \mathcal{H}^{\boldsymbol{u}} _0.
	\end{equation}
	On one hand, $\nu$ is bounded, the operator $-\operatorname{div}\left(\nu(\theta_b) \mathbb{D}(\boldsymbol{u})\right): \boldsymbol{\mathcal{H}}^{\boldsymbol{u}}_0 \rightarrow (\boldsymbol{\mathcal{H}}^{\boldsymbol{u}}_0)^{\prime}$ is linear and continuous, and $\boldsymbol{u} \in \mathrm{L}^2(0, T; \boldsymbol{\mathcal{H}}^{\boldsymbol{u}}_0)$; this implies that $-\operatorname{div}\left(\nu(\theta_b) \mathbb{D}(\boldsymbol{u})\right) \in \mathrm{L}^2\left(0, T ; (\boldsymbol{\mathcal{H}}^{\boldsymbol{u}}_0)^{\prime}\right)$.
	On the other hand, $\boldsymbol{F} \in \mathrm{L}^2(0, T; \boldsymbol{L}^{2}(\Omega_b))$. 
	By Lemma \ref{pro-trilinear}, we have established that $b(\boldsymbol{u}, \boldsymbol{u}, \boldsymbol{\psi})=\langle B(\boldsymbol{u}), \boldsymbol{\psi}\rangle$ is trilinear and continuous on $\boldsymbol{\mathcal{H}}^{\boldsymbol{u}}_0$, and $\|B(\boldsymbol{u})\|_{(\boldsymbol{\mathcal{H}}^{\boldsymbol{u}}_0)^{\prime}} \leq \|\boldsymbol{u}\|_{\boldsymbol{\mathcal{H}}^{\boldsymbol{u}} _0}^2$. Thus, $B(\boldsymbol{u}) \in \mathrm{L}^1\left(0, T, (\boldsymbol{\mathcal{H}}^{\boldsymbol{u}}_0)^{\prime}\right)$.
	Consequently,  $\partial_t \boldsymbol{u} \in \mathrm{L}^1\left(0, T; (\boldsymbol{\mathcal{H}}^{\boldsymbol{u}}_0)'\right)$.\\
	In similar ways,  the second equation of \eqref{equ33} can also be expressed by:
	\begin{equation}
		\sum_{i=b,ts}\left\langle\partial_t\theta_{i}, S_{i} \right\rangle=-\sum_{i=b,ts}a_{\theta_{i}}(\theta_{i} ; \theta_{i}, S_{i} )-d(\boldsymbol{u}, \theta_b, S_{b})
		+\sum_{i=b,ts}c_{\varphi_{i}}(\theta_{i} ,\varphi_{i},S_{i})\, \text{ for all } S_{i}\in \mathcal{H}^{\theta_i}_0(\Omega_i).
	\end{equation}
	Since Assumption \eqref{assump-1} is satisfied and $\theta_i \in L^{2}(0,T ; \mathcal{H}^{\theta_i}_0(\Omega_i))$, we deduce that $a_{\theta_{i}}(\theta_{i} ; \theta_{i}, \cdot )\in L^{2}(0,T ; L^{2}(\Omega_i))$. Furthermore, we collect the result that $\boldsymbol{u} \in \mathrm{L}^2(0, T; \boldsymbol{\mathcal{H}}^{\boldsymbol{u}}_0)$ and $\varphi_i \in L^{2}(0,T ; \mathcal{H}^{\varphi_i}_0)\cap L^{\infty}(0,T; L^{\infty}(\Omega_i))$ to get $d(\boldsymbol{u}, \theta_b, \cdot)$ and $c_{\varphi_{i}}(\theta_{i} ,\varphi_{i},\cdot)$ are bounded in $L^{1}(0,T ; H^{-1}(\Omega_b))$ and $L^{2}(0,T ; H^{-1}(\Omega_i))$, respectively. Consequently, $\partial_t\theta_i$ is bounded in $L^{1}(0,T ; (\mathcal{H}^{\theta_{i}})^{\prime})$ for $i=b,\,ts$.
	
	To introduce the pressure $\pi$, we set
	$$
	V(t)=\int_0^t \left(\nu(\theta_b) \mathbb{D}(\boldsymbol{u})\right)(s)\, d s,\, R(t)=\int_0^t(\boldsymbol{u} \cdot \nabla) \boldsymbol{u}(s) \,d s,\, \text{and} \, K(t)=\int_0^t \boldsymbol{F}(s)\,d s .
	$$ 
	It is clear that $V, K, R \in C\left(0, T ;\left(\mathrm{H}^1(\Omega)\right)^{\prime}\right)$. Integrating \eqref{eq:u-derivative} over $[0, t]$ yields
	$$
	\left\langle \boldsymbol{u}(t)-\boldsymbol{u}_0- \operatorname{div} V(t)+R(t)+K(t), \boldsymbol{\psi}\right\rangle=\mathbf{0} \text { for all } t \in[0, T] \text { and for all } \boldsymbol{\psi} \in \boldsymbol{\mathcal{H}}^{\boldsymbol{u}}_0 \text {. }
	$$
	By application of the Rham theorem \cite{R.Temam2001}, we find, for each $t \in[0, T]$, the existence of some function $P(t) \in \mathrm{L}_0^2(\Omega_b)$ such that
	$$
	\boldsymbol{u}(t)-\boldsymbol{u}_0-\operatorname{div}V(t)+R(t)+K(t)+\nabla P=\mathbf{0},
	$$
	where $\mathrm{L}_0^2(\Omega)=\left\{w \in \mathrm{L}^2(\Omega_b), \int_{\Omega} w \mathrm{~d} \boldsymbol{x}=0\right\}$. Therefore, $\nabla P \in C\left(0, T ; \mathrm{H}^{-1}(\Omega_b)\right)$, and thus $P \in C\left(0, T ; \mathrm{L}_0^2(\Omega_b)\right)$. By derivation with respect to $t$ in the sense of distributions, we obtain
	$$
	\partial_t \boldsymbol{u}-\operatorname{div}(\nu(\theta_b) \mathbb{D}(\boldsymbol{u}))+(\boldsymbol{u} \cdot \nabla) \boldsymbol{u}+\boldsymbol{F}+\nabla \pi=\mathbf{0},
	$$
	where $\pi=\partial_t P \in W^{-1, \infty}\left(0, T ; L_0^2(\Omega_b)\right)$.
}

\section{Numerical simulations}\label{numerical}  
The aim of this section is to validate our proposed model. For this, we provide three tests based on the existing data in the literature, for instance, \cite{materiel2,materiel1,materiel3}. We consider a $\Omega$ domain which is decomposed by two sub-domains $\Omega_b$ and $\Omega_{ts}$ as described in Figure \ref{Domain} where  $L= 1.5$, $H= 1$ and $r= 0.075$. 
In our numerical simulations, the electrode thickness is assumed negligible because their diameter is very small.    
In addition, the electrical conductivities $\sigma_i$, thermals conductivities $\eta_i$, and blood conductivity $\nu$ have been modeled as temperature-dependent functions and are given by the following equations 
\begin{eqnarray*}
\sigma_b(\theta_b)&=&\left\{\begin{array}{ll}
\sigma_0 \exp\left(0.015 (\theta_b - \bar{\theta})\right) & \text { for } \theta_b \leq 99^{\circ} \mathrm{C}, \\
2.5345 \sigma_0 & \text { for } 99^{\circ} \mathrm{C}<\theta_b \leq 100^{\circ} \mathrm{C}, \\
2.5345 \sigma_0\left(1- 0.198\left(\theta_b-100^{\circ} \mathrm{C}\right)\right) & \text { for } 100^{\circ} \mathrm{C}< \theta_b \leq 105^{\circ} \mathrm{C}, \\
0.025345 \sigma_0 & \text { for } \theta_b >105^{\circ} \mathrm{C},
\end{array}\right .
\\
\eta_b(\theta_b) &=&\left\{\begin{array}{ll}
\eta_0+ 0.0012 \left(\theta_b -  \bar{\theta}\right) & \text { for } \theta_b \leq 100^{\circ} \mathrm{C}, \\
\eta_0+0.0012\left(100^{\circ} \mathrm{C}-\bar{\theta}\right) & \text { for } \theta_b>100^{\circ} \mathrm{C},
\end{array}\right .
\\\sigma_{ts}(\theta_{ts})&=&
\sigma_0+ 0.02 \left(\theta_{ts} -  \bar{\theta}\right),  \\
\eta_{ts}(\theta_{ts})&=&\eta_0+0.0012\left(\theta_{ts}-\bar{\theta}\right), 
\end{eqnarray*}
where $\sigma_0 = 0.6$ and  $\eta_0 = 0.54$ are the constant electrical conductivity and
the thermal conductivity, respectively, at core body temperature, $\bar{\theta}=37^{\circ} \mathrm{C}$ and $\varphi_d=1$.
Moreover, the viscosity and density of blood are $0.0021 \mathrm{~Pa} \cdot \mathrm{s}$ and $1000 \mathrm{~kg} / \mathrm{m}^3$, respectively, while those of saline are $0.001 \mathrm{~Pa} \cdot \mathrm{s}$ and $1000 \mathrm{~kg} / \mathrm{m}^3$, respectively, based on the material property of water.  

Let now give the time discretization of our model. First, we define a time subdivision $t_0= 0 < \cdots < t_M =T$, where $M$ is an integer and the time steps is as follows $\tau_n = {t_{n+1}-t_n}, n=0, \cdots, M-1$. For the space discretization, we adopt a finite element method. Specifically, we exploit the finite element $P1-$Bubble to compute the values of the velocity variable and the $P1$ finite element to approximate the  unknowns of pressure, the temperatures and potentials. Herein, we keep the same notations of the variables $\mathbf u$, $\pi$, $\theta_i$ and $\varphi_i$ for the discrete versions.

We mention that the reformulation of the studied model into an algebraic system of differential equations is important. Indeed, it allows using a time lag scheme. That is to say, given the solution of the heat equations at the previous instant, we then solve the decoupled Navier-Stokes equations and potentials equations for the time step $n-1$ as
$$
	\begin{aligned}
		&\left\{\begin{array}{rclll}
			\boldsymbol{u}_{t}-\nabla \cdot(\nu( {\theta_b}^{n-1}) \mathbb{D}(\boldsymbol{u}))+\nabla \cdot(\boldsymbol{u} \otimes \boldsymbol{u})+\nabla \pi & =&\boldsymbol{F}({\theta_b}^{n-1})& \hbox { in } \Omega_{T,1}, \\
			\nabla \cdot \boldsymbol{u} & =&0 &\hbox{ in } \Omega_{T,1}, \\
			\boldsymbol{u} & =&\boldsymbol{u}_d &\hbox{ on }  \Sigma_b^D,
			\\
			\boldsymbol{u}(\boldsymbol{x}, 0) & =&\boldsymbol{u}_{0}(\boldsymbol{x}) &\hbox { in } \hbox { }\Omega_b ,
		\end{array}\right . \\
	\hbox{ and}&
	\\
	&	\left\{	\begin{array}{rclll}
			- \operatorname{div}(\sigma_i({\theta_i}^{n-1}) \nabla \varphi_i) & = &0 & \hbox { in } &\Omega_{T,i}, \\
			(\sigma_i({\theta_i}^{n-1})\nabla \varphi_i) \cdot \boldsymbol{n}_i & = & 0 & \hbox { on } &\Sigma_i^{N}, \\
			\varphi_i & = & \varphi_{i,d} & \hbox { on } &\Sigma^{F}.
		\end{array}\right. 
	\end{aligned}
$$
Thus, we obtain the potentials $\varphi_{b}^{n-1}$ and $\varphi_{ts}^{n-1}$,  the velocity $\boldsymbol{u}^{n-1}$ and the the pressure  $\pi^{n-1}$ at the time step $n-1$. This allows to solve the temperature equation at time $n$. \\
We mention here an interesting question which is how to treat the temperature advection-diffusion equation.
Obviously, not all discretization of this equation are equally stable without regularization techniques. 
For this reason, we can use discontinuous elements which is more efficient for pure advection problems. 
However, in the presence of diffusion terms, the discretization of the Laplace operator is cumbersome due to the large number of additional terms that must be integrated on each face between the cells.
Consequently, a better alternative is to add some nonlinear viscosity   $\tilde {\eta}(\theta)$ to the model that only acts in the vicinity of shocks and other discontinuities. 
The viscosity $\tilde {\eta}(\theta)$ is chosen in such a way that if $\theta$ satisfies the original equations, the additional viscosity is zero. 
In our case, we will opt for the stabilization strategy developed in \cite{guermond2011entropy} 
that builds on a suitably defined residual and a limiting procedure for the additional viscosity. 
For this, let us define a residual $R_\alpha(\theta)$ as follows:
$$
R_\alpha(\theta)=\left(\frac{\partial \theta}{\partial t}+\mathbf{u} \cdot \nabla \theta-\nabla \cdot \eta(\overline \theta) \nabla \theta 
      -\sigma  ( \overline \theta) |\nabla \varphi|^2 \right) \theta^{\alpha-1}, \quad \alpha \in [1,2].
$$ 
Note that $R_\alpha(\theta)$ will be zero if $\theta$ satisfies the temperature equation. Multiplying terms out, we obtain the following entirely equivalent form:
$$
R_\alpha(\theta)=\frac{1}{\alpha} \frac{\partial\left(\theta^\alpha\right)}{\partial t}+\frac{1}{\alpha} \mathbf{u} \cdot \nabla\left(\theta^\alpha\right)-\frac{1}{\alpha} \nabla \cdot \eta(\overline \theta) \nabla\left(\theta^\alpha\right)+\eta(\overline \theta)(\alpha-1) \theta^{\alpha-2}|\nabla \theta|^2-\gamma \theta^{\alpha-1}.
$$
Using the above equation, we define the artificial viscosity as a piece-wise constant function defined on each cell $K$ with diameter $h_K$ separately giving by:
$$
\left.\tilde {\eta}_\alpha(\theta)\right|_K=\beta\|\mathbf{u}\|_{L^{\infty}(K)} \min \left\{h_K, h_K^\alpha \frac{\left\|R_\alpha(\theta)\right\|_{L^{\infty}(K)}}{c(\mathbf{u}, \theta)}\right\},
$$
where, 
$\beta$ is a stabilization constant and 
$\displaystyle c(\mathbf{u}, \theta)=c_R\|\mathbf{u}\|_{L^{\infty}(\Omega)} \operatorname{var}(\theta)|\operatorname{diam}(\Omega)|^{\alpha-2}$ where
$\operatorname{var}(\theta)=\max _{\Omega} \theta-\min _{\Omega} \theta$ is the range of present temperature values  and $c_R$ is a dimensionless constant. \\
If on a particular cell the temperature field is smooth, then we expect the residual to be small and the stabilization term that injects the artificial diffusion will be rather small, when no additional diffusion is needed. 
In addition, if we are on or near a discontinuity in the temperature field, then the residual will be large and the artificial viscosity will ensure the stability of the scheme.

%

Finally, we consider the following boundary conditions. A velocity 
$\boldsymbol{u}=\boldsymbol{u}_e = \left(  \begin{array}{l} 4y(H-y) \\ 0 \end{array}\right)$ on boundary $\Sigma_1$, 
$\boldsymbol{u} =\boldsymbol{u}_s$ on boundary $\Sigma_8$, and we change the boundary condition of the viscosity in  the boundary $\Sigma_3$  that is $-\pi \boldsymbol{n}_{b}+\nu(\theta_b)\mathbb{D}(\boldsymbol{u}) \boldsymbol{n}_{b} =\mathbf{0}$.  On boundaries $\Sigma_i, i= 2,  7$,  we assume that the velocity is zero i.e $\boldsymbol{u} =\left(  \begin{array}{l} 0 \\ 0 \end{array}\right)$ on $\Sigma_{2}\cup\Sigma_{7}$.
We set on the boundaries $\Sigma_i$,  $i=1, \cdots,6$ the temperatures, $\theta_b=\theta_{ts}=37^{\circ} \mathrm{C}$.
For the potential equation, we fix $\varphi_{i}=\varphi_{d}=1$ on $\Sigma_8$ and the homogeneous Dirichlet condition in the remaining boundaries. In the following, we provide three numerical experiments where the aims are to show the influences of the saline flow and the external force.

\begin{figure}
	\begin{minipage}[t]{0.49\linewidth}
		\centering
		\includegraphics[width=\linewidth,valign=t]{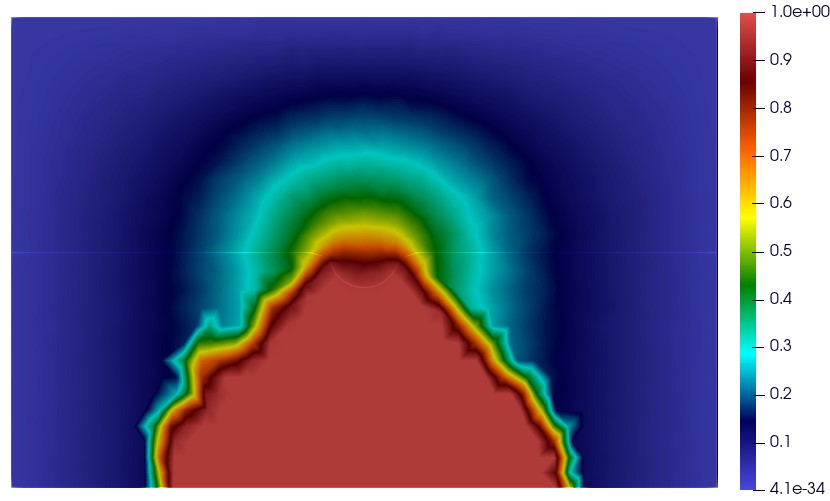}
	\end{minipage}
	\vspace{0.01\textwidth}
	\begin{minipage}[t]{0.49\linewidth}
		\centering
		\includegraphics[width=\linewidth,valign=t]{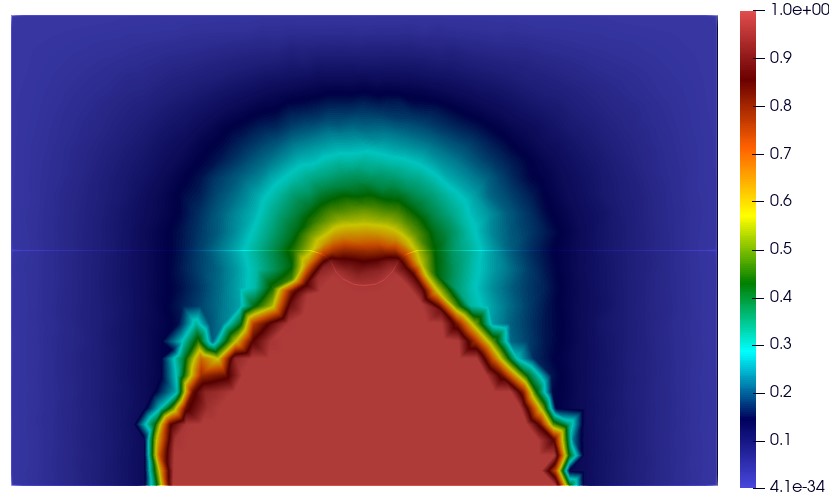}
	\end{minipage}
	\begin{minipage}[t]{0.49\linewidth}
		\centering
		\includegraphics[width=\linewidth,valign=t]{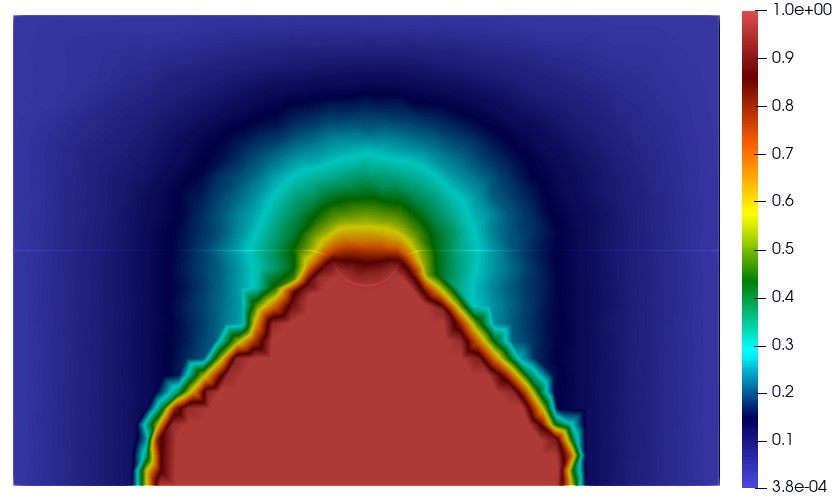}
	\end{minipage}
	\vspace{0.01\textwidth}
	\begin{minipage}[t]{0.49\linewidth}
		\centering
		\includegraphics[width=\linewidth,valign=t]{figs/TST1/1PHIT3.jpg}
	\end{minipage}
	\caption{Snapshot of evolution of the potentials 
		at four time moments $t=\frac{T}{6}$, $\frac{T}{2}$,   $\frac{3T}{4}$, $T$. } 
	\label{fig:test10}
\end{figure}
\subsection{Test 1: heat transfer and blood flow }\label{test1}
The aim of this test is to show that our model is validated by comparing with the results from the literature. First of all, we mention that the impact of saline viscosity and the external forces are neglected, which is expressed as $\boldsymbol{u}_s=0$ and $\textbf{F}=0$. In this scenario, the continuity condition is used in place of the saline heat and the Dirichlet condition of viscosity on the boundary of $\Sigma_3$, which are  $(\eta_b(\theta_b) \nabla \theta_b) \cdot \boldsymbol{n}_b=-(\eta_{ts}(\theta_{ts})\nabla \theta_{ts}) \cdot \boldsymbol{n}_{b}$ on $\Sigma_8$  and $\boldsymbol{u}=\boldsymbol{u}_e $ on $\Sigma_3$, respectively.  

The first remark is that the computed potential evolves very slowly during the time iterations, see Figure \ref{fig:test10}. This can be justified by the fact that the only data in the potential equation is the source $\varphi_i$ which is constant and the electrical conductivity $\sigma(\theta) = \sigma_0 \exp(0.015(\theta-\bar{\theta}))$. Thus, we omit the figures of the potential as there is no significant change during the iterations.
 
Now, let comment on the mechanism of how and why we got our numerical results. The applied potential increases the temperature of the tissue part near to the catheter and then diffusing in cardiac tissue domain, see Figure \ref{fig:test11}. Thanks to the the continuity on the boundaries, the temperature of the blood domain is also increased, this can be justified by the presence of the electric field term $-\sigma_{ts}(\theta_{ts})\nabla \varphi_{ts}\cdot \nabla \varphi_{ts} $. It is important to mention that this temperature is also influenced by the presence of the blood. More precisely, it is transported in the direction of the fluid because of the presence of the term $\boldsymbol{u}\cdot\nabla \theta_b$. On the other hand, it is clear that the speed of the fluid and the pressure in the fluid part are influenced by the viscosity which depends on the temperature $\theta_b$. More precisely, we see a flow recirculation zone, with a local enhance of velocity upstream to the catheter. While it remains almost uniform in the other regions far from the catheter at time $T/6$. In time progress, we can see that the velocity and the pressure are influenced by the presence of the viscosity, see the second column of Figure \ref{fig:test11}. 

To cut it short, this test shows the influence of each term in our model, it shows the coupling effects, that is to say, the fluid, the pressure, the temperatures, and the potentials are affecting each other. 

There is close agreement between the model and the experimental
results from the literature. 
\begin{figure}
	\begin{minipage}[t]{0.49\linewidth}
		\centering
		\includegraphics[width=\linewidth,valign=t]{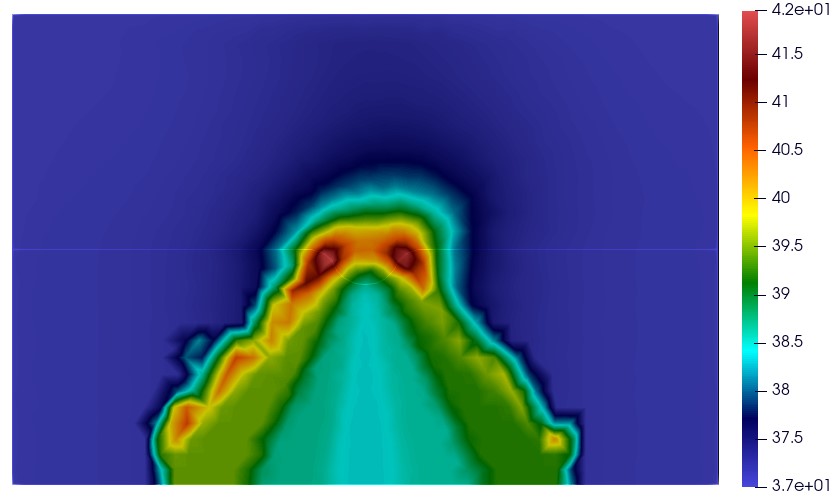}
	\end{minipage}
	\vspace{0.01\textwidth}
	\begin{minipage}[t]{0.49\linewidth}
		\centering
		\includegraphics[width=\linewidth,height=2.9cm,valign=t]{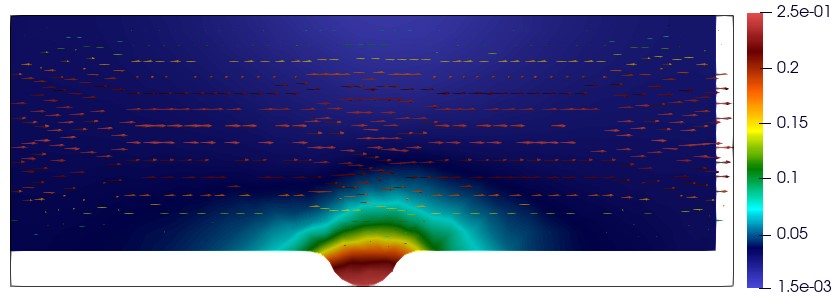}
	\end{minipage}
	\begin{minipage}[t]{0.49\linewidth}
		\centering
		\includegraphics[width=\linewidth,valign=t]{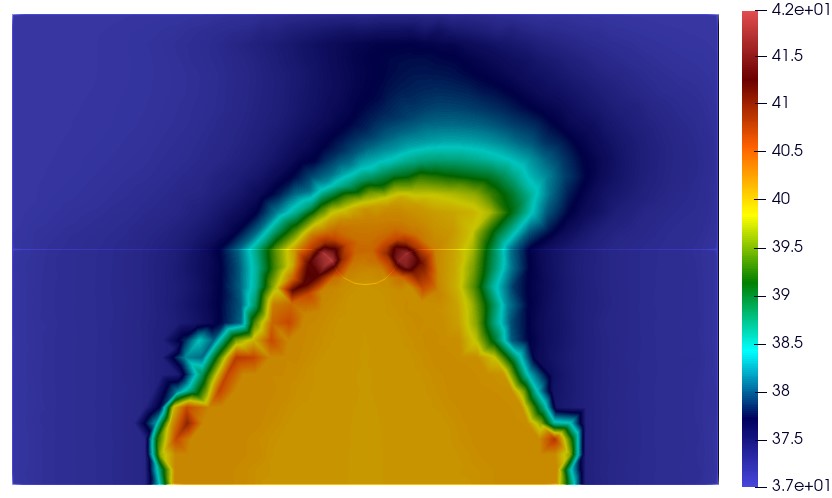}
	\end{minipage}
	\vspace{0.01\textwidth}
	\begin{minipage}[t]{0.49\linewidth}
		\centering
		\includegraphics[width=\linewidth,height=2.9cm,valign=t]{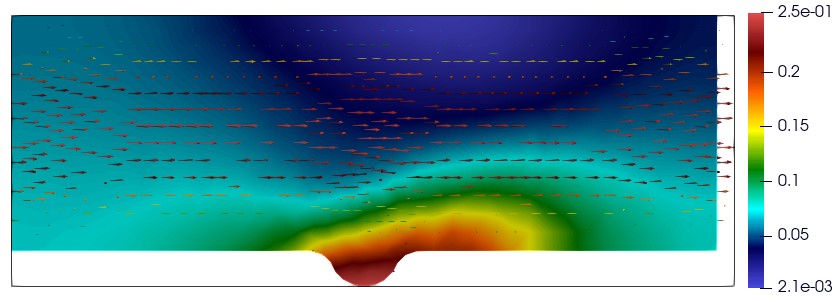}
	\end{minipage}
	\begin{minipage}[t]{0.49\linewidth}
		\centering
		\includegraphics[width=\linewidth,valign=t]{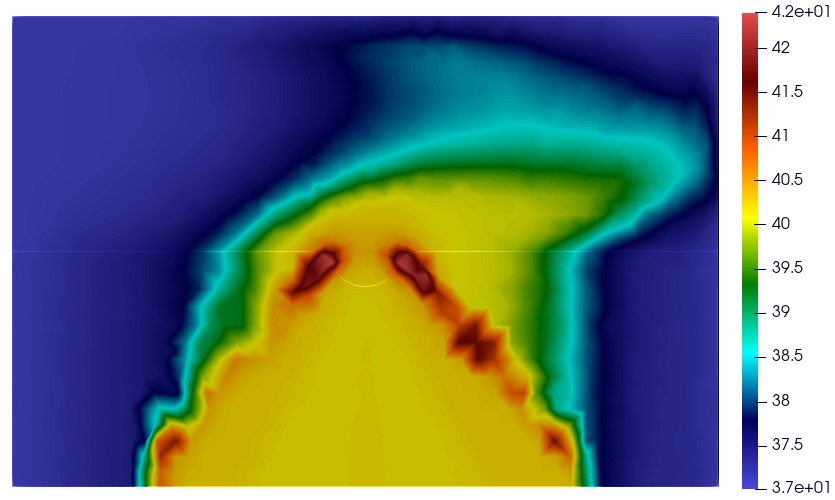}
	\end{minipage}
	\vspace{0.01\textwidth}
	\begin{minipage}[t]{0.49\linewidth}
		\centering
		\includegraphics[width=\linewidth,height=2.9cm,valign=t]{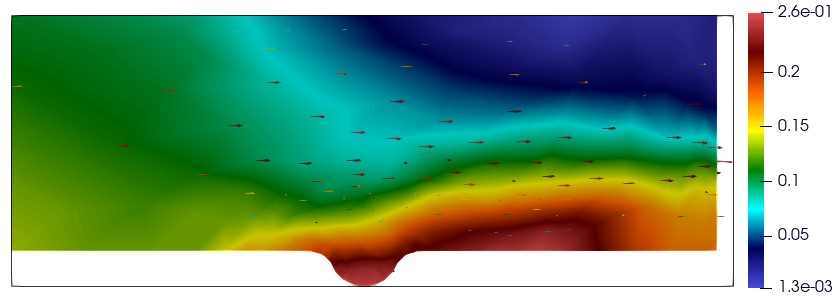}
	\end{minipage}
	\caption{Test 1 : evolution of heats (column 1), velocity and pressure (column 2) 
		at three time moments $t=\frac{T}{6}$ (line 1), 
		$t= \frac{T}{2}$ (line 2) and $t=T$ (line 3). } 
	\label{fig:test11}
\end{figure}

\subsection{Test 2: the saline flow effect}
In this test, we aim to demonstrate the effect of the saline flow. As it is shown in Test $1$ (cf. Subsection \ref{test1}), the temperature in the neighborhoods of the catheter achieves a critical values between $40^{\circ} \mathrm{C}$ and $42^{\circ} \mathrm{C}$. Thus, naturally it is necessary to cool down this zone and make it's temperature down. For that, we need to inject a fluid where the saline heat is $20^{\circ} \mathrm{C}$ i.e $\theta_b=\theta_{ts}=\theta_s=20^{\circ} \mathrm{C}$ on $\Sigma_8$. This can be done by  considering the following 
$$\boldsymbol{u} =\boldsymbol{u}_s= \left(\begin{array}{l}  \displaystyle \frac{20}{r}(x-\frac L 2  +r)(\frac L 2 +r -x)(\frac L 2   -x) \\ \displaystyle  \frac{-20}{r} (x-\frac L 2  +r)(\frac L 2 +r -x)y \end{array}\right)$$
 on boundary $\Sigma_8$. To present these evolution, we show in Figure \ref{fig:test2} the results of numerical simulations at three different times $t = \frac{T}{6},\, \frac{T}{2},\, T $, where each row of the figure represents the corresponding time in the same order. In the first column, we show the heats $\theta_b$ and $\theta_{ts}$, and in the second column, we show the velocity field and the pressure. Clearly, we notice that the injected saline flow $\textbf{u}_s$ diminishes the calculated heats. This leads to the possibility of cooling the domain by the saline fluid from $\Sigma_{8}$. In addition, we observe the rotation of the fluid in the areas subject to heat variations, especially in the area near the outlet boundary $\Sigma_3$.

\begin{figure}[pos=!ht]
	\begin{minipage}[t]{0.49\linewidth}
		\centering
		\includegraphics[width=\linewidth,valign=t]{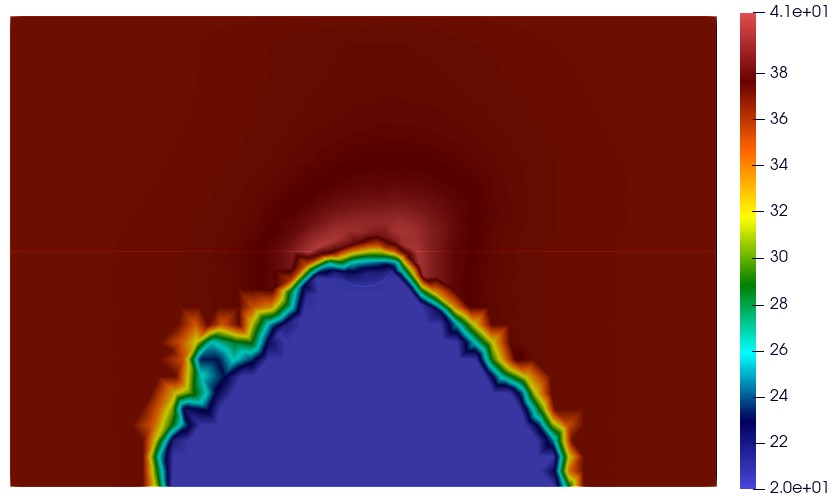}
	\end{minipage}
	\vspace{0.01\textwidth}
	\begin{minipage}[t]{0.49\linewidth}
		\centering
		\includegraphics[width=\linewidth,height=2.9cm,valign=t]{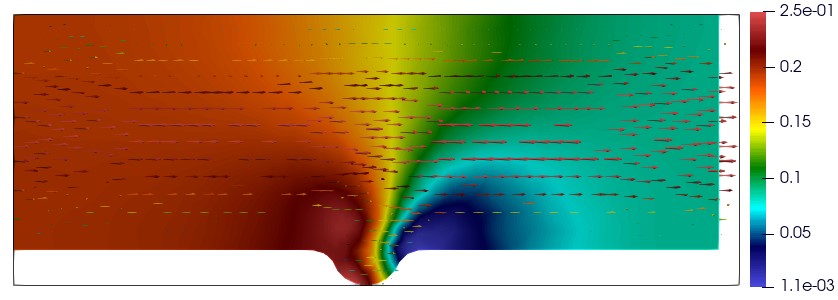}
	\end{minipage}
	\begin{minipage}[t]{0.49\linewidth}
		\centering
		\includegraphics[width=\linewidth,valign=t]{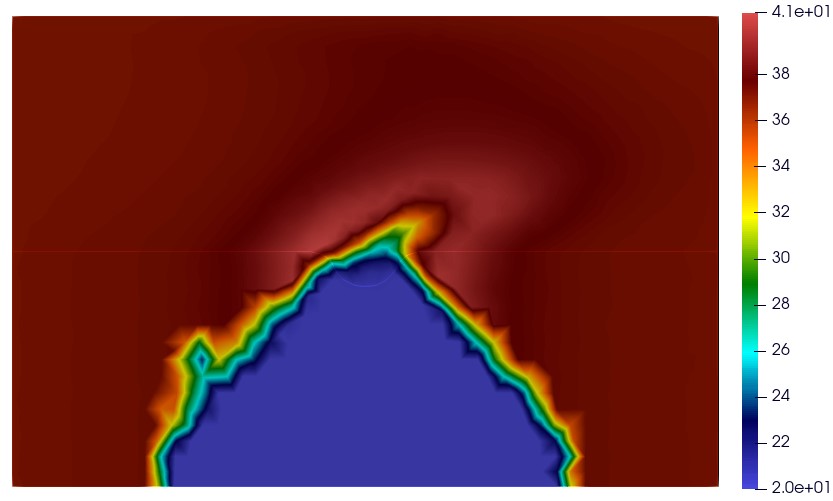}
	\end{minipage}
	\vspace{0.01\textwidth}
	\begin{minipage}[t]{0.49\linewidth}
		\centering
		\includegraphics[width=\linewidth,height=2.9cm,valign=t]{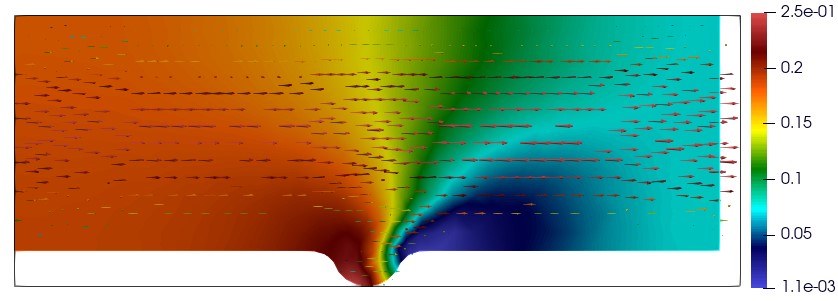}
	\end{minipage}
	\begin{minipage}[t]{0.49\linewidth}
		\centering
		\includegraphics[width=\linewidth,valign=t]{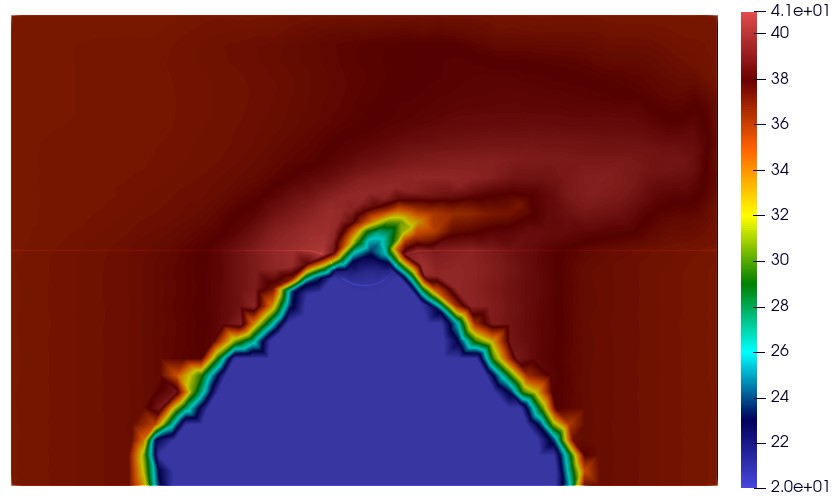}
	\end{minipage}
	\vspace{0.01\textwidth}
	\begin{minipage}[t]{0.49\linewidth}
		\centering
		\includegraphics[width=\linewidth,height=2.9cm,valign=t]{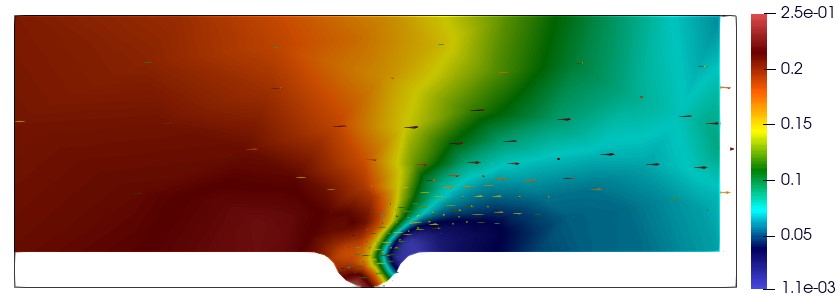}
	\end{minipage}
	\caption{Test 2 : evolution of heats (column 1), velocity and pressure (column 2)
		at three time moments $t=\frac{T}{6}$ (line 1), 
		$t= \frac{T}{2}$ (line 2) and $t=T$ (line 3). } 
	\label{fig:test2} 
\end{figure}
\subsection{Test 3: the external force effect} 
In final test, we are interested in the behavior of the heats when the fluid source term is non-zero. Thus, we consider the fluid source $\textbf{F}= - \left(\begin{array}{l}  0\\ 10^{-3}  9.81/303 \left( \theta_b - \bar{\theta} \right)   \end{array}\right)$ as in Boussinesq equations.  Figure \ref{fig:test3} represents the evolution of the heats $\theta_i$ (first column) and of the velocity and pressure (column 2) at times $t=\frac{T}{6}$, $\frac{T}{2},$ and $T$. 
We notice the rotation of the fluid in the areas subject to heat variations, especially in the area near the outlet boundary $\Sigma_3$.
This is justified by the structure of the source term $\textbf{F}$, in particular the term $\theta_b-\bar{\theta}$. Indeed by the principle of maximum, the velocity changes its sign according to the value of the temperature $\theta_b$ whether it is lower or higher than $\bar{\theta}$. Consequently, we achieve a reduction of the temperature in certain areas of the domain.  
 \begin{figure}
 	\begin{minipage}[t]{0.49\linewidth}
 		\centering
 		\includegraphics[width=\linewidth,valign=t]{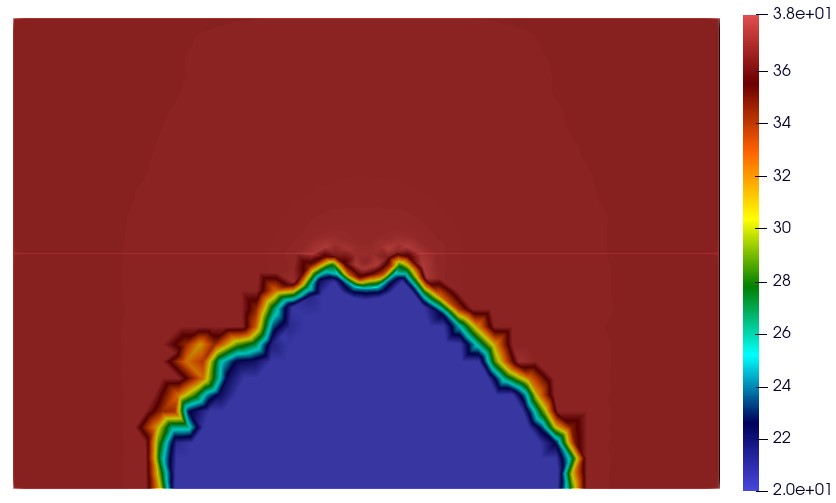}
 	\end{minipage}
 	\vspace{0.01\textwidth}
 	\begin{minipage}[t]{0.49\linewidth}
 		\centering
 		\includegraphics[width=\linewidth,height=2.9cm,valign=t]{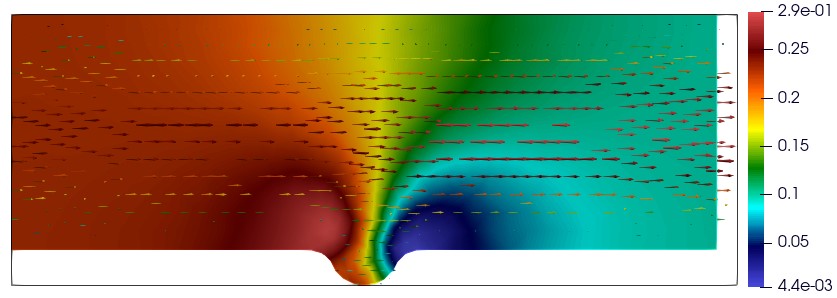}
 	\end{minipage}
 	\begin{minipage}[t]{0.49\linewidth}
 		\centering
 		\includegraphics[width=\linewidth,valign=t]{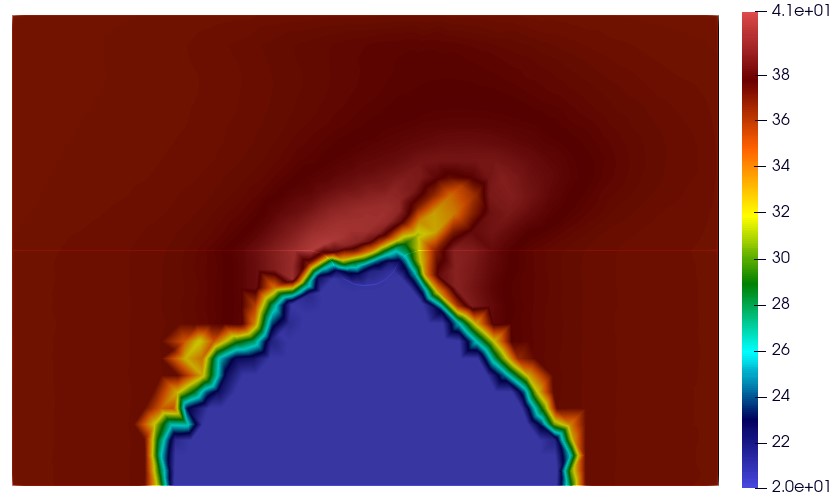}
 	\end{minipage}
 	\vspace{0.01\textwidth}
 	\begin{minipage}[t]{0.49\linewidth}
 		\centering
 		\includegraphics[width=\linewidth,height=2.9cm,valign=t]{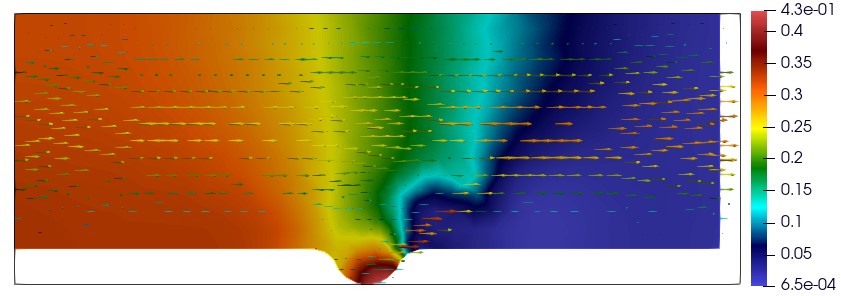}
 	\end{minipage}
 	\begin{minipage}[t]{0.49\linewidth}
 		\centering
 		\includegraphics[width=\linewidth,valign=t]{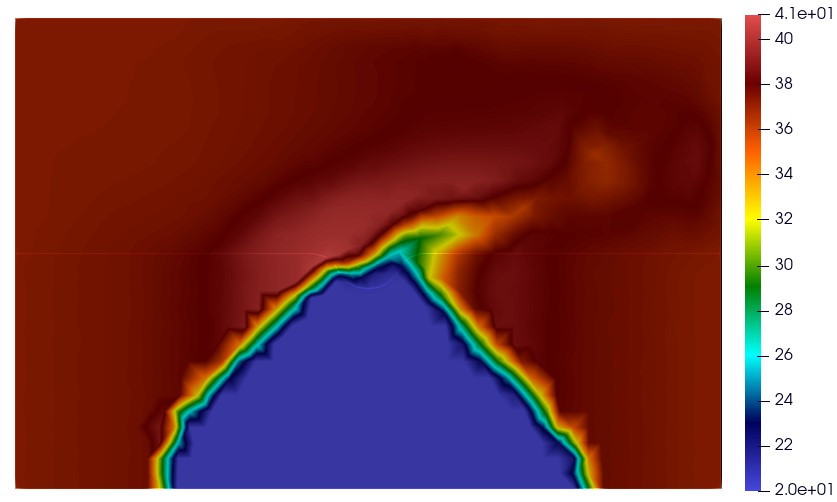}
 	\end{minipage}
 	\vspace{0.01\textwidth}
 	\begin{minipage}[t]{0.49\linewidth}
 		\centering
 		\includegraphics[width=\linewidth,height=2.9cm,valign=t]{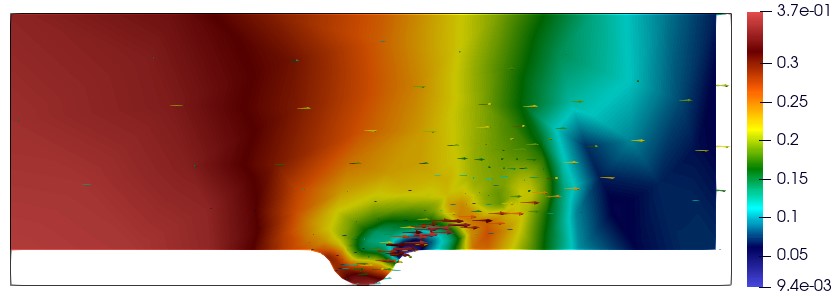}
 	\end{minipage}
 	\caption{Test 3 : evolution of heat (column 1), velocity and pressure (column 2) 
 		at three time moments $t=0$ (line 1), 
 		$t= \frac{T}{2}$ (line 2) and $t=T$ (line 3). }
 	\label{fig:test3}
 \end{figure}
\section{Conclusion and perspectives}\label{conclusion}
In this paper, a new coupled electro-thermo-fluid RFA model describing radiofrequency ablation phenomena in cardiac tissue and Newtonian fluid medium governed by the incompressible Navier-Stokes has been proposed. The proposed model can be considered as an improved model in \cite{B23}. Indeed, their model consists of a coupled thermistor and the incompressible Navier-Stokes equations that describe the evolution of temperature, velocity, and  potential only in the blood vessel. 

In this paper, the existence of the global solutions has been proved using the Faedo-Galerkin with a priori estimates and compactness arguments. In addition, numerical simulations in different cases have been illustrated in a two-dimensional space using the finite element method. 

We propose a general mathematical framework that is capable of including the following features: a radiofrequency ablation phenomenon for cardiac tissue, predicting the temperature of the tissues, the saline flow, and the external force. These features are demonstrated in our numerical simulations. We believe this work can lead to interesting perspectives, such as optimal control models and inverse problems, namely the identification of the frequency factor of different tissue types. In addition, other perspectives consist of improving the bidomain model by adding the stochastic effects, for instance in the incompressible Navier-Stokes equation, see \cite{BTZ22}.  
 
\bibliographystyle{cas-model2-names} 
\bibliography{BibFile}

\end{document}